\documentclass[11pt,leqno]{article}
\usepackage{amsthm,amsfonts,amssymb,epsfig,graphics,amsmath,eufrak}
\usepackage[latin1]{inputenc}\relax

\newcommand{\sL}{{L_*^\eps}}

\newcommand{\csL}{{\cL_*^\eps}}
\newcommand{\csLd}{{(\cL_*^\eps)^\dagger}}
\newcommand{\csLe}{{\cL_*^{\eps,\eta}}}

\newcommand{\sgn}{{\text{\rm sgn}}}

\font\tenronde=rsfs10
\font\sevenronde=rsfs7
\font\fiveronde=rsfs5
\newfam\rondefam
\scriptscriptfont\rondefam=\fiveronde
\textfont\rondefam=\tenronde
\scriptfont\rondefam=\sevenronde
\def\ronde{\fam\rondefam\tenronde}

\newcommand{\sS}{{\ronde S}}

\renewcommand{\Re}{\text{\rm Re }}

\newcommand{\const}{\text{\rm constant}}

\def\g{\gamma}

\newcommand{\RR}{{\mathbb R}}

\newcommand{\mA}{{\mathbb A}}

\newcommand\cB{{\mathcal  B}} 

\newcommand\cU{{\mathcal  U}}

\newcommand\cV{{\mathcal  V}}

\newcommand\cR{{\mathcal  R}}

\newcommand\cL{{\mathcal  L}}
\newcommand\cN{{\mathcal  N}}

\newcommand\cF{{\mathcal  F}}

\newcommand\cT{{\mathcal T}}
\newcommand\cS{{\mathcal S}}

\newcommand{\mez}{{\frac{1}{2}}}

\def\eps{\varepsilon }
\def\blockdiag{\text{\rm blockdiag}}

\def\D{\partial }
\newcommand\adots{\mathinner{\mkern2mu\raise1pt\hbox{.}
\mkern3mu\raise4pt\hbox{.}\mkern1mu\raise7pt\hbox{.}}}

\newcommand{\Id}{{\rm Id }}
\newcommand{\im}{{\rm Im }\, }
\newcommand{\re}{{\rm Re }\, }

\newcommand{\la}{\langle }
\newcommand{\ra}{\rangle }

\newtheorem{theo}{Theorem}[section]
\newtheorem{prop}[theo]{Proposition}
\newtheorem{cor}[theo]{Corollary}
\newtheorem{lem}[theo]{Lemma}

\newtheorem{ass}[theo]{Assumption}

\newtheorem{rem}[theo]{Remark}

\numberwithin{equation}{section}

\begin{document}

\title{Existence of semilinear relaxation shocks}

\author{\sc \small Guy M\'etivier\thanks{
IMB, Universit\'e de Bordeaux,
33405 Talence Cedex, France; metivier@math.u-bordeaux.fr.,
},
Kevin Zumbrun\thanks{Indiana University, Bloomington, IN 47405;
kzumbrun@indiana.edu:
K.Z. thanks the University of Bordeaux I
for its hospitality during the visit in which
this work was carried out.
Research of K.Z. was partially supported
under NSF grants number DMS-0070765 and DMS-0300487.
 }}


\maketitle

\begin{abstract}
We establish existence with sharp rates of decay and distance
from the Chapman--Enskog approximation
of small-amplitude shock profiles of a class of
semilinear relaxation systems including 
discrete velocity models obtained
from Boltzmann and other kinetic equations.
Our method of analysis is based on the macro--micro decomposition introduced
by Liu and Yu for the study of Boltzmann profiles,
but applied to the stationary rather than the time-evolutionary equations.
This yields a simple proof by contraction mapping in weighted $H^s$ spaces.
\end{abstract}

\tableofcontents


\section{Introduction}\label{intro}

We consider the problem of existence of relaxation profiles 
\begin{equation}\label{relaxprof}
U(x,t)=\bar U(x-st), \quad \lim_{z\to \pm \infty}\bar U(z)=U_\pm
\end{equation}
of a semilinear relaxation system 
\begin{equation}\label{grelax}
U_t +F(U)_x= Q(U), 
\end{equation}
in one spatial dimension,   with the following structure: 

\begin{ass}\label{1.1}

(H1)  The flux $F$ is linear in $U$, so that  $F(U) = A U$ for some constant matrix $A$. 

(H2)  There are linear coordinates  such that 
\begin{equation}\label{block}
\quad U= \begin{pmatrix} u\\v \end{pmatrix},
\quad Q=\begin{pmatrix} 0\\q\end{pmatrix},
\end{equation}
$u\in \RR^n$, $v\in \RR^r$. 

(H3) $q$ has nondegenerate equilibria parametrized by $u$; more precisely there are a smooth function   
 $v_*$ from $  \RR^n$ to $ \RR^r$ and $\theta > 0$ such that for all $u$:
\begin{equation}\label{qassum1}
q(u,v_*(u))=  0, \qquad  
\Re \sigma (\partial_v q(u,v_*(u)))\le -\theta, 
\end{equation}
$\sigma(\cdot )$ denoting spectrum  and $\Re$ the real part.
\end{ass}

Common examples are discrete kinetic models obtained
by discrete velocity or other approximation from continuous kinetic models
such as Boltzmann or Vlasov--Poisson equations; 
for example, Broadwell and other lattice gas models [PI].
Other examples are the semilinear relaxation schemes introduced by
Jin--Xin [JX] and Natalini [N]
for the purpose of numerical approximation of hyperbolic systems.
Here, we are thinking particularly of the case $n$ bounded
and  $r>>1$ arising through discretization of the Boltzmann
equation, or the case $r \to \infty$ arising in Boltzmann itself;
that is, {\it we seek estimates and proof independent of the dimension
of $v$.}

For fixed $n$, $r$, the existence problem has been treated in
\cite{YZ, MaZ1} under the additional assumption 
\begin{equation}\label{nondeg}
\det (dF-sI)\ne 0
\end{equation}
corresponding to nondegeneracy of the traveling-wave ODE.
However, as pointed out in \cite{MaZ2,MaZ3}, this assumption is
unrealistic for large models, and in particular is not satisfied for
the Boltzmann equations, for which the eigenvalues of $dF$ are
constant particle speeds of all values.
Our goal here, therefore, is to revisit the existence problem
{\it without the assumption \eqref{nondeg}}, with the
eventual aim being to establish a simple proof of existence of
small-amplitude Boltzmann profiles.
Of course, existence of such was established some time ago
in \cite{CN}; however, the proof is rather complicated, involving
detailed resolvent estimates in weighted $L^\infty$ spaces in spatial
and velocity variables, and
so it seems of use to seek a simpler approach based on weighted
$L^2$ spaces and standard energy estimates.

Our method of analysis is motivated by the ``macro-micro decomposition''
technique introduced by Liu and Yu \cite{LY}, 
in which fluid (macroscopic, or equilibrium) and transient
(microscopic) effects are separated and estimated by different techniques.
This was used in \cite{LY} to show
by a study of the time-evolutionary equations
that the Boltzmann profile constructed in \cite{CN} 
has nonnegative probability density,
that is, to show positivity of Boltzmann profiles
{\it assuming that such a profile exists}.

Our approach here is very much in the spirit of that of \cite{LY},
based on approximate Chapman--Enskog expansion combined with
Kawashima type estimates (the macro--micro decomposition
of the reference),
but carried out for the {\it stationary} (traveling-wave) rather than the
time-evolutionary equations, and estimating the finite-dimensional
fluid part using sharp ODE estimates in place of the 
energy estimates of \cite{LY}.
In this latter part, we are much aided by the more favorable properties
of the stationary fluid equations, a rather standard boundary
value ODE system, as compared to the time-evolutionary equations, a 
hyperbolic--parabolic system of PDE.

Our main result is to show existence with sharp rates of decay and distance
from the Chapman--Enskog approximation of small-amplitude quasilinear relaxation
shocks in the general case that the profile ODE may become degenerate.
See Sections \ref{model} and \ref{CEapprox} for model assumptions and 
description of the Chapman--Enskog approximation, and 
Section \ref{results} for a statement of the main theorem.
In the present, semilinear case, 
a simple contraction-mapping argument suffices;
the quasilinear case is treated by Nash--Moser iteration in \cite{MeZ1}.
In \cite{MeZ2}, we show that the argument of this paper carries
over with minor modifications to the infinite-dimensional Boltzmann equation 
with hard potential to yield existence of small-amplitude Boltzmann shock 
profiles, recovering and slightly sharpening the results of \cite{CN}.
This in a sense completes the analysis of \cite{LY}, providing by
a common set of techniques both existence (through the present argument)
and (through the argument of \cite{LY}) positivity.
At the same time it gives a truly elementary 
proof of existence of Boltzmann profiles.

Finally, we note that spectral stability has been shown for general
small-amplitude quasilinear relaxation profiles in \cite{MaZ3}, without the
assumption \eqref{nondeg}, under the assumption that the profile
exist and satisfy exponential bounds like those of the viscous case.
The results obtained here verify that assumption, completing the 
analysis of \cite{MaZ3}.
It would be very interesting to continue along the same lines to
obtain a complete nonlinear stability result as in \cite{MaZ1},
in particular for Boltzmann shocks.

Existence results in the absence of condition
\eqref{nondeg} have been obtained in special cases in \cite{MaZ5,DY}
by quite different methods
(for example, 
center-manifold expansion near an assumed single degenerate point \cite{DY}).
However, the decay bounds as stated,
though exponential, are not sufficiently
sharp with respect to $\eps$ for the needs of \cite{MaZ3}.
More important, the techniques used in these analyses do not appear to 
generalize to the infinite-dimensional (e.g., Boltzmann) case.

\section{Model, assumptions, and the reduced system}\label{model}

Taking without loss of generality $s=0$, we study the traveling-wave
ODE
\begin{equation}\label{relax}
AU'= Q(U),
\end{equation}
\begin{equation}\label{relaxform}
\quad  U=\begin{pmatrix} u\\v \end{pmatrix},
\quad A\equiv \const,
\quad Q=\begin{pmatrix} 0\\q(u,v)\end{pmatrix}
\end{equation}
governing solutions of  \eqref{relaxprof}, where $q$ satisfies 
\eqref{qassum1}. 
 
We use the notations  
\begin{eqnarray}
\label{A}
 A= \begin{pmatrix} A_{11}& A_{12}\\A_{21}& A_{22}\end{pmatrix},
\\\label{f}
f(u,v):= A_{11}u + A_{12}v.
\end{eqnarray}

We make the standard assumption of {\it simultaneous symmetrizability}
\cite{Y}: 
 
\begin{ass}\label{SS}
(SS) \quad  There exists a smooth,  symmetric and uniformly positive
definite matrix $S(U)$ such that 

\quad i)  for all $U$, $S(U)A$ is symmetric, 

\quad ii) for all equilibria $ U_*= (u, v_*(u))$,
$S\,dQ(U_*)$ is symmetric  nonpositive with
\begin{equation}\label{rank}
\dim \ker SdQ=\dim \ker dQ \equiv n.
\end{equation}

\end{ass}\label{GC}

 We  also make the Kawashima assumption of {\it genuine coupling} \cite{K}:
 
 \begin{ass} 
 \label{assGC}
(GC) \quad  For all equilibria $ U_*= (u, v_*(u))$,  there exists no eigenvector of 
$A$ in the kernel of $dQ(U_*)$.
Equivalently, given Assumption \ref{SS} (see \cite{K}), 
there exists in a neighborhood $\cN$ of the
equilibrium manifold a skew symmetric $K=K(U)$ such that 
\begin{equation}\label{K}
\Re (KA-SdQ)(U) \ge \theta>0
\end{equation}
for all $U\in \cN$.
\end{ass}
Recall from \cite{Y} (see aslo Section~\ref{secCE}), that the reduced, Navier--Stokes type equations
obtained by Chapman--Enskog expansions are
\begin{equation}\label{ce} 
f_*(u)'= (b_*(u)u')',
\end{equation}
where
\begin{eqnarray}
\label{cevalues} 
&f_*(u)&:=f(u,v_*(u))=A_{11}u+A_{12}v_*(u),
\\
\label{bstar}
&b_*(u)&:= -A_{12} c_* (u) 
\end{eqnarray}
with
\begin{equation}
\label{cstar}
\begin{aligned}
c_* (u) := \partial_v  &q^{-1} (u,v_*(u))&\\
&\Big(A_{21} +  A_{22}d v_*(u)    - d v_*(u) (A_{11}+ A_{12} d v_*(u) 
\Big). 
\end{aligned}
\end{equation}
 Note also, by the Implicit Function Theorem, that 
$$
d v_*(u) = -   \partial_v q^{-1}\partial_u q (u, v_*(u)). 
$$
 For the reduced system \eqref{ce}, 
simultaneous symmetrizability becomes:
\medbreak
(ss)\quad
There exists $s(u)$ symmetric positive
definite such that $s\, df_*$ is symmetric and $sb_*$ is
symmetric positive semidefinite.
\medbreak
\noindent 
We have likewise a notion of genuine coupling \cite{K}:

\medbreak
(gc)\quad 
There is no eigenvector of $df_*$ in $\ker b_*$.
\medbreak

We note first the following important observation of \cite{Y}.

\begin{prop}[\cite{Y}]\label{redconds}
Let \eqref{relax} as described above be a symmetrizable
system satisfying the genuine coupling condition (GC).
Then, the reduced system \eqref{ce} is a symmetrizable
system satisfying genuine coupling condition (gc).
\end{prop}

\begin{proof}  We  only give here a sketch, mentioning the key points and refereeing 
  e.g. to  \cite{MaZ3} for details.  Fixing $ U$, one is reduced to  constant matrices and linear algebra, 
  with matrices 
  $$
  A = \begin{pmatrix} A_{11}& A_{12}\\A_{21}& A_{22}\end{pmatrix},\quad 
  Q = \begin{pmatrix} 0& 0\\Q_{21}& Q_{22}\end{pmatrix}, 
  $$
 and symmetrizer $S$.  With 
 $$
 P = \begin{pmatrix} \Id & 0 \\ V & \Id \end{pmatrix}, \qquad V = - Q_{22}^{-1} Q_{21}, 
 $$
 the change of unknowns $ U = P \widetilde U$ transforms the problem to an equivalent one 
 with matrices $\widetilde A = P^{-1} A P$,  $\widetilde Q = P^{-1} Q P$ and symmetrizer 
 $\widetilde S = P^* S P $, with  
 $$
  \widetilde Q = \begin{pmatrix} 0& 0\\0 & Q_{22}\end{pmatrix}. 
 $$
Therefore $\widetilde S$ is  block diagonal  ($\widetilde S_{21} = 0$ and $ \widetilde S_{12} = 0$), 
$\widetilde S_{11} \widetilde A_{11} $ is symmetric, 
 $ \widetilde S_{11} \widetilde A_{12} = ( \widetilde S_{22} \widetilde A_{21})^*$ and 
 $\widetilde S_{22} \widetilde Q_{22}$ is definite negative. 
 Next, the associated matrix $\widetilde b$ is :  
 $$
  \widetilde b  = 
 - \widetilde A_{12} \widetilde Q_{22}^{-1} \widetilde A_{21} . 
 $$
 Thus $\widetilde S_{11}$ is definite positive, symmetrizes $\widetilde A_{11}$ and 
 $$
 \widetilde S_{11}  \widetilde b  = 
 -   \widetilde S_{11}  \widetilde A_{12} (\widetilde S_{22} \widetilde Q_{22}^{-1})^{-1}  S_{22}  \widetilde A_{21} 
  = 
 -  (\widetilde S_{22}  \widetilde A_{21} )^*(\widetilde S_{22} \widetilde Q_{22}^{-1})^{-1} 
 \widetilde S_{22}  \widetilde A_{21} 
 $$
 is symmetric and nonnegative.   
Noticing that 
 \begin{eqnarray}
 \label{chvarA}
&  \tilde A_{11}     &= A_{11} + A_{12} V,  
 \\
 \label{chvarb}
 & \widetilde b &   =     
  - A_{12} Q_{22}^{-1} \big(  A_{21} + A_{22} V  - V ( A_{11} + A_{12} V) \big)   =  b_* 
\end{eqnarray}
 this implies that the property (ss) is satisfied with symmetrizer 
 $s = \widetilde S_{11}$ (In terms of  the original matrices, 
 $s  = S_{11} +  S_{12} V  = S_{11} +  S_{12} V + V^* (S_{21} +  S_{22} V)  $, 
 since $ S_{21} +  S_{22} V = 0 $ as a consequence of  the block diagonal structure of 
 $\widetilde S$).  
 
 Similarly, the property (GC) is transported to the system $(\widetilde A, \widetilde Q)$, meaning that 
 $$
 \widetilde A_{11} u = \lambda u, \quad \widetilde A_{21} u = 0 \qquad \Rightarrow \qquad u = 0. 
 $$
 The symmetry property of $\widetilde S_{11} \widetilde b$ implies that 
 $$
 \ker \widetilde b = \ker \widetilde  A_{21} 
 $$
 and the property  (gc) immediately follows.
\end{proof}

Besides the basic properties guaranteed by Proposition \ref{redconds},
we assume that the reduced system satisfy the following important
additional conditions.

\begin{ass}\label{goodred}

(i) The matrix $b_*(u)$ has constant left kernel.

(ii) For all values of $u$,  $\ker \pi_* df_*(u) \cap \ker b_*(u) = \{ 0 \}$,   
where $\pi_*(u)$ is the zero eigenprojection associated with $b_*(u)$.   

\end{ass}

The importance of Assumption \ref{goodred} in the present situation
is that it ensures that the zero-speed profile problem for the reduced system,
\begin{equation}\label{vprof}
f_*(u)'=(b_*(u)u')', \quad \lim_{z\to \pm \infty} u(z)=u_\pm 
\end{equation}
or, after integration from $-\infty$ to $x$,
\begin{equation}\label{intvprof}
b_*(u)u'= f_*(u)-f_*(u_\pm),
\end{equation}
may be expressed as a nondegenerate ODE in $u_2$, coordinatizing
$u=(u_1,u_2)$ with $u_1=\pi_*u$ and $u_2=(I-\pi_*)u$.

Condition (i) was pointed out in \cite{Ze}, condition (ii) in
\cite{MaZ3,Z1,GMWZ}.

\begin{rem}\label{couplingrmk}
\textup{
Assumption \ref{goodred} is also central to the linearized stability
analysis of general Navier--Stokes type equations in \cite{Z2, GMWZ}.
It appears to be independent from the genuine coupling conditions
(GC), (gc), except in the special case that $u$ is scalar, for which
(GC), (gc) reduce to $\ker b_*=\emptyset$. 
It is satisfied for the important example of Boltzmann
equations, which is our main motivation.
}
\end{rem}

Next, we assume that the classical theory of    weak shocks 
can be applied to  \eqref{vprof}, assuming that the flux $f_*$ has a  genuinely nonlinear 
eigenvalue near $0$:

\begin{ass}\label{profass}
 
In a neighborhood $\cU_*$ of   a given base state $u_0$,  
$df_*$ has a simple eigenvalue $\alpha$ near zero, with $\alpha (u_0) = 0$, and such that the
associated hyperbolic characteristic field is genuinely 
nonlinear, i.e., after a choice of orientation, $\nabla \alpha \cdot r(u_0) <  0$, where
$r$ denotes the eigendirection associated with $\alpha$.
 
\end{ass}

\begin{rem}
\textup{Assumption \ref{profass} is standard,
and is satisfied in particular for the compressible
Navier--Stokes equations resulting from Chapman--Enskog approximation
of the Boltzmann equation. 
Assumptions \ref{SS} and \ref{GC} are verified in \cite{Y} 
for a wide variety of discrete kinetic models.
Assumptions \ref{goodred} and \ref{profass} on the reduced equations
must be checked in individual cases.
}
\end{rem}


\section{Chapman--Enskog approximation} \label{CEapprox}
\label{secCE}

Integrating the first equation of \eqref{relax} and 
noticing that the end states $(u_\pm, v_\pm) $ must be equilibria and thus satisfy 
$v_\pm = v_*(u_\pm)$,  we obtain
\begin{equation}\label{intprof}
\begin{aligned}
A_{11} u + A_{12} v &=  f_*(u_\pm),\\
A_{21}u'+ A_{22}v'&=q(u,v).
\end{aligned}
\end{equation}

Because $f$ is linear,  the first equation reads
\begin{equation}\label{T1}
f_*(u) + A_{12}(v-v_*(u))  = f_*(u_\pm).
\end{equation}
The idea of Chapman--Enskog approximation is that 
$v -  v_*(u)$ is small  (compared to the fluctuations $u - u_\pm$). Taylor expanding the second equation, we obtain
$$
\begin{aligned}
(A_{21}+A_{22}dv_*(u)) u' + A_{22}(v-v_*(u))'&=
\partial_v q(u,v_*(u))(v-v_*(u)) 
\\
&\quad + O(|v-v_*(u)|^2),    
\end{aligned}
$$
or inverting $\partial_v q$  
\begin{equation}\label{T2}
\begin{aligned}
v-v_*(u)  =  & \ \partial_v q^{-1} (u, v_*(u))   \big(A_{21} + A_{22} d v_*(u) \big) 
 u' \\
& + O(|v-v_*(u)|^2) + O(|(v-v_*(u))'|). 
\end{aligned}
\end{equation}
The derivative of \eqref{T1} implies that 
$$
 \big( A_{11} u + A_{12} dv_* (u) \big) u'  =  O(|(v-v_*(u))'|) . 
$$
 Therefore,  \eqref{T2} can be replaced by 
 \begin{equation}\label{T2b}
\begin{aligned}
v-v_*(u)&= c_* (u) u' 
   + O(|v-v_*(u)|^2) + O(|(v-v_*(u))'|),
\end{aligned}
\end{equation}
where  $c_*$ is defined at \eqref{cstar}. 
Substituting in \eqref{T1},  we thus obtain the approximate viscous profile ODE
\begin{equation}\label{approxprof}
\begin{aligned}
b_*(u)u'&= f_*(u) -f_*(u_\pm) 
+ O(|v-v_*(u)|^2) + O(|(v-v_*(u))'|),
\end{aligned}
\end{equation}
where $b_*$ is as defined in \eqref{bstar}.
\begin{rem}
\label{rem31}  \textup{ 
The above calculation leaves a great deal
of flexibility in the choice of   $b_*$ satisfying \eqref{approxprof},  namely  it is only
specified modulo multiples of }
$$
  A_{11} + A_{12} d v_* (u), 
$$
\textup{as we used when passing from \eqref{T2} to \eqref{T2b}. However,  
we have chosen to use the standard   definition \eqref{cevalues} of $b_*$ because 
it is the natural choice which is invariant by change of variables (see \eqref{chvarb}) 
and it  is known
by Proposition \ref{redconds} and by the explicit example of Boltzmann
to have good properties.  
But, it might be, for example, that a different representative
could be strictly parabolic, so slightly easier to handle.
This seems to be just a curiosity, as the analysis is already sufficient
to treat the standard case.
But, it is interesting from the viewpoint of the Chapman--Enskog
expansion and possible alternative representations. }
\end{rem}

Motivated by \eqref{T2}--\eqref{approxprof}, we define an approximate
solution $(\bar u_{NS}, \bar v_{NS})$ of \eqref{intprof} by choosing 
$\bar u_{NS}$  as a solution of 
\begin{equation}
\label{NS}
b_*(\bar u_{NS})\bar u_{NS}' = f_*(\bar u_{NS}) -f_*(u_\pm),
\end{equation}
and $\bar v_{NS}$  as the first approximation given by \eqref{T2} 
\begin{equation}
\label{NSv}
\begin{aligned}\bar v_{NS} -v_*(\bar u_{NS})  =    c_* (\bar u_{NS}) 
 \bar u_{NS}'.
\end{aligned}
\end{equation}

Small amplitude shock profiles   solutions of \eqref{NS}  are constructed 
using the center manifold  analysis of \cite{Pe}
under conditions (i)-(ii) of Assumption \ref{goodred}; see discussion  in 
\cite{MaZ5}.

\begin{prop}\label{NSprofbds} Under Assumptions~\ref{profass} and \ref{goodred}, 
in a neighborhood of 
$(u_0, u_0)$ in $\RR^n \times \RR^n$, 
there is a smooth  manifold $\cS$ of dimension $n$  passing through $(u_0, u_0)$,  such that 
for $(u_-, u_+) \in \cS$ with   amplitude $\eps:=|u_+ -u_-| > 0$ 
sufficiently small, and direction $(u_+-u_-)/\eps $ sufficiently close
to $r(u_0)$,   the zero speed shock profile equation   \eqref{NS} has  a unique (up to translation) 
solution   $\bar u_{NS}$ in $\cU_*$. 
The shock profile is necessarily of {\rm Lax type}: i.e., with
dimensions of the unstable subspace of $df_*(u_-)$
and the stable subspace of $df_*(u_+)$ summing to one plus the
dimension of $u$, that is $n+1$.

Moreover, 
there is  $\theta>0$ and for all $k$ there is $C_k $ independent of $(u_-, u_+) $ and $\eps$,   
such that 
\begin{equation}\label{NSbds}
|\partial_x^k (\bar u_{NS}-u_\pm)|\le C_k \eps^{k+1}e^{-\theta \eps|x|},
\quad x\gtrless 0. 
\end{equation}
\end{prop}

We denote by 
 $\cS_+$  the set  of $(u_-, u_+) \in \cS $  with  amplitude $\eps:=|u_+ -u_-| > 0$ 
sufficiently small  and direction $(u_+-u_-)/\eps $ sufficiently close
to $r(u_0)$ such that the profile $\bar u_{NS}$ exists.

Given $(u_-, u_+) \in \cS_+  $ with associated profile $\bar u_{NS}$, 
we define $\bar v_{NS} $ by \eqref{NSv} and 
    \begin{equation}
    \label{NSU}
    \bar U_{NS} := (\bar u_{NS}, \bar v_{NS}).
    \end{equation}
    It  is an approximate solution of \eqref{intprof} in the following sense: 

\begin{cor}\label{redbds}
For   $(u_-, u_+) \in \cS_+$, 
\begin{equation} 
\label{exactfeq}
 A_{11} \bar u_{NS}  + A_{12} \bar v_{NS} - f_*(u\pm) = 0
 \end{equation} 
 and 
 $$
 \cR_v  := A_{21}\bar u_{NS}' + A_{22}\bar v_{NS}'-q(\bar u_{NS},\bar v_{NS})  
$$ 
satisfies 
\begin{equation}\label{eq:resbds}
|   \D_x^k \cR_v  (x) |  \le  C_k \eps^{k+3}e^{-\theta \eps|x|} , \quad x\gtrless 0
\end{equation}
where $C_k$   is  independent of $(u_-, u_+) $ and  $\eps=|u_+ -u_-|$. 
 \end{cor}

\begin{proof}
Given the  choice of $\bar v_{NS}$, the first equation is a rewriting of the profile equation  
\eqref{NS}.  

Next, note that 
$$
 \bar v_{NS}  - v_* (\bar u_{NS}) = O  (| \bar u'_{NS} | ) , 
 \quad  \big( \bar v_{NS}  - v_* (\bar u_{NS})\big)'   = O  (| \bar u''_{NS} | ) + O  (| \bar u'_{NS} |^2 ), 
$$
where here  $O ( \cdot )$ denote  smooth functions of $\bar u_{NS}$ and its derivatives, 
which vanish as indicated. With similar notations, 
the Taylor expansion of $q$ and the definition of $\bar v_{NS}$ show that 
$$
\begin{aligned}
\cR_v       = &  O(| \bar v_{NS} -v_*(\bar u_{NS})|^2) + O(|(\bar v_{NS} -v_*(\bar u_{NS}))'|)
\\
& +   d v_* (\bar u_{NS}   \big(A_{11} + A_{12}  dv_* (\bar u_{NS}) \big )  \bar  u'_{NS}  . 
\end{aligned}
$$
Moreover, 
$$
\begin{aligned}
\big(A_{11} + A_{12}  d v_* (\bar u_{NS}) \big)  \bar  u'_{NS}  &= \big(  f_* (\bar u_{NS}) \big)' 
=   \big( b_* (\bar u_{NS}) \bar u'_{NS}) \big)'   
   \\
   & = O(|\bar u_{NS}'|^2) +    O(|\bar u_{NS}''|). 
\end{aligned}
$$
 This implies that  
 $$
\cR_v =   O(|\bar u_{NS}'|^2) +    O(|\bar u_{NS}''|).  
$$
  satisfies the estimates stated in \eqref{eq:resbds}. 
\end{proof}

\begin{rem}\label{correctorrmk}
\textup{
One may check that 
if we did not include the correction from equilibrium on the righthand
side of \eqref{NSv}, taking instead the simpler prescription
$\bar v_{NS} =v_*(\bar u_{NS})$ as in \cite{LY}, then 
the residual error that would result in \eqref{exactfeq} would be
too large for our later iteration scheme to close.
This is a crucial difference between our analysis and the analysis
of \cite{LY}.
}
\end{rem}

 %

\section{Statement of the main theorem}\label{results}

We are now ready to state the main result.
Define a base state $U_0=(u_0,v_*(u_0))$ and a
neighborhood $\cU=\cU_*\times \cV$.   

\begin{theo}\label{main}
Let Assumptions (SS), (GC), and  \ref{goodred} hold on the
neighborhood $\cU$ of $U_0$, with $Q\in C^{\infty}$. 
Then, there are $\eps_0 > 0$  and 
$\delta > 0$ such that for $(u_-, u_+) \in \cS+$ with  amplitude $\eps:=|u_+-u_-| \le \eps_0$,   the standing-wave equation 
\eqref{relax} has a solution   
$\bar U$ in $\cU$, 
 with associated Lax-type 
equilibrium shock $(u_-,u_+)$, satisfying for all $k  $: 
\begin{equation}\label{finalbds}
\begin{aligned}
\big|\partial_x^k (\bar U- \bar U_{NS})\big|
&\le C_k \eps^{k+2}e^{-\delta  \eps|x|},\\
|\partial_x^k (\bar u-u_\pm)|&\le C_k \eps^{k+1}e^{-\delta \eps|x|},
\quad x\gtrless 0,\\
\big|\partial_x^k (\bar v-v_*(\bar u)\big|
&\le C_k \eps^{k+2}e^{-\delta  \eps|x|},\\
\end{aligned}
\end{equation}  
where $\bar U_{NS}=(\bar u_{NS}, \bar v_{NS})$ is the 
approximating Chapman--Enskog profile defined in \eqref{NSU}, and
$C_k$ is independent of  $\eps$. 
Moreover, up to translation, this solution is unique
within a ball of radius $c\eps$ about $\bar U_{NS}$ in norm 
$\|\cdot\|_{L^2}+\eps^{-1}\|\D_x \cdot\|_{L^2}
+ \eps^{-2}\|\D_x^2 \cdot\|_{L^2} $, for $c>0$ sufficiently small.
(For comparison, $\bar U_{NS}-U_\pm$ is order $\eps^{1/2}$ in this norm,
by \eqref{finalbds}(ii)--(iii).)
\end{theo}
 
Bounds \eqref{finalbds} show that (i) the behavior of profiles
is indeed well-described by the Navier--Stokes approximation,
and (ii) profiles indeed satisfy the exponential decay rates
required for the proof of spectral stability in \cite{MaZ3}.
From the second observation, we obtain immediately from
the results of \cite{MaZ3} the following stability result,
partially generalizing that of \cite{LY} 
for the Boltzmann equations.\footnote{
Liu and Yu prove the stronger result of 
linearized stability with respect to zero-mass perturbations
that are sufficiently small in an appropriate norm.}

\begin{cor}[\cite{MaZ3}]\label{MaZcor}
Under the assumptions of Theorem \ref{main}, the resulting
profiles $\bar U$ are spectrally stable for amplitude $\eps$
sufficiently small, in the sense that the linearized operator
$L:= \partial_x A(\bar U) -dQ(\bar U)$ about $\bar U$
has no $L^2$ eigenvalues
$\lambda$ with $\Re \lambda \ge 0$ and $\lambda \ne 0$.
\end{cor}

The remainder of the paper is devoted to the proof of Theorem \ref{main}.

\begin{rem}
\textup{Theorem \ref{main} yields uniqueness only among solutions close
to the Chapman--Enskog approximant $\bar U_{NS}$.
The stability result of Liu--Yu \cite{LY} should give uniqueness
among solutions in a ball of small but $O(1)$ radius, assuming that
they have zero relative mass compared to $\bar U_{NS}$.
Indeed, it should be possible to upgrade this to general-mass perturbations
to obtain ultimately a full $O(1)$ uniqueness result.
Stability with respect to general-mass perturbations is an important
open problem.}
\end{rem}


\section{Outline of the proof}\label{outline}
 
\subsection{Nonlinear perturbation equations}
Defining the perturbation variable $U:= \bar U- \bar U_{NS}$,
and expanding about $\bar U_{NS}$,
we obtain from \eqref{intprof} the nonlinear perturbation equations
\begin{eqnarray}
  A_{11} u + A_{12} v  & = &0 
\\
  A_{21} u' + A_{22} v' -    d q (\bar U_{NS}) U & = & - \cR_v + N( U) 
\end{eqnarray} 
where the remainder $N(U) $ is  a smooth function of $U_{NS} $ and $U$, 
vanishing at second order at  $U =0$: 
\begin{equation}\label{Nbds}
N(U)= \cN(\bar U_{NS}, U)  = O(|U|^2). 
\end{equation}
We push the reduction a little further, using that 
\begin{equation}
\label{defM}
M := 
dq(\bar u_{NS}, \bar v_{NS})-
dq(\bar u_{NS}, v_*(u_{NS})) =
O(|\bar v_{NS}-v_*(\bar u_{NS})|). 
\end{equation}
 Therefore the equation reads 
 \begin{equation}
 \label{intpert}
 \begin{aligned}
\cL_*^\eps U:=&
\begin{pmatrix}0 & 0 \\ A_{21} & A_{22} 
\end{pmatrix}U'
+
\begin{pmatrix}A_{11} & A_{12} \\
- Q_{21}  & - Q_{22} \end{pmatrix}U
\\ 
= &
\begin{pmatrix} 0  \\ - \cR_v + MU + N(U)\end{pmatrix}
\end{aligned}
\end{equation}
where 
\begin{equation}
Q_{21} = \D_u q (\bar u_{NS}, v_* (\bar u_{NS})), 
\quad 
Q_{22} = \D_v q (\bar u_{NS}, v_* (\bar u_{NS})). 
\end{equation}

Differentiating the first line, it implies that 
\begin{equation}\label{pert} 
L_*^\eps U:=
AU'-dQ(\bar u_{NS},v_*(\bar u_{NS}))U=
\begin{pmatrix}   0  \\ - \cR_v + MU + N(U)\end{pmatrix}. 
\end{equation}
  
The linearized operator $A\partial_x - dQ(\bar U)$
about an exact solution $\bar U$ of the profile equations
has kernel $\bar U'$, by translation invariance, so is not invertible.
Thus, the linear operators $L_*^\eps $ and $\csL$ are not 
expected to be invertible,
and we shall see later that they are not.  
Nonetheless, one can check that $\csL$ is surjective in Sobolev spaces and    define a right inverse
$\csL \csLd\equiv I$, or solution operator
$(\cL_*^\eps)^\dagger$ of the equation 
\begin{equation}
\label{neweq}
\cL_*^\eps U=F:=  \begin{pmatrix}f\\g\end{pmatrix}, 
\end{equation}  
as recorded by   Proposition~\ref{invprop} below. 
Note that $\sL$ is not surjective   because   the first equation requires 
a zero mass condition on the source term. This is why we solve 
the integrated equation \eqref{intpert} and not \eqref{pert}. 

To  define the partial inverse $\csLd$, we  specify  one solution of 
\eqref{neweq} by adding the co-dimension one  internal 
condition: 
\begin{equation}\label{phasecond}
\ell _\eps \cdot u(0)  =0 
\end{equation}
where $\ell_\eps$ is a certain unit vector to be specified below. 

\begin{rem}\label{ellchoice}
\textup{ There is a large flexibility in the choice of $\ell_\eps$.
Conditions like \eqref{phasecond}  are  known to fix the indeterminacy in the 
resolution of the linearized profile equation  from \eqref{NS}
and  it remains well adapted in the present context,  see section~\ref{linCEestimates} below. 
A possible choice,  would be to choose $\ell_\eps$ independent of 
$\eps$ and parallel to the left  eigenvector of $ df_* (u_0)$ for the eigenvalue $0$
(see  Assumption~\ref{profass}).    }
\end{rem}



\subsection{Fixed-point iteration scheme}

The coefficients  and the error term $\cR_v$ are smooth functions of 
$\bar u_{NS}$ and its derivative, thus behave like smooth functions of 
$ \eps x$. Thus, it is natural to solve the equations in spaces which reflect 
this scaling. We do not introduce explicitly the change of variables
$\tilde x = \eps x$, but introduce norms which correspond to the usual $H^s$ norms 
in the $\tilde x $ variable : 
\begin{equation}
\label{defnorm}
\|f \|_{H^s_\eps} =  
\eps^{\mez}  \|f \|_{L^2}+
\eps^{-\mez }\|\partial_x f\|_{L^2}+ \dots + 
\eps^{\mez-s}\|\partial_x^s f\|_{L^2}.
\end{equation}
We also introduce weighted spaces and norms, which encounter for the exponential 
decay of the source and solution: introduce the notations.  
\begin{equation}
\label{modx}
<x>:= (x^2+1)^{1/2}
\end{equation}
For  $\delta \ge 0$ (sufficiently small), we denote by $H^s_{\eps, \delta}$ the space of 
functions $f$ such that   $ e^{\delta  \eps <x>} f \in H^s$ equipped with the norm
\begin{equation}
\label{defwnorm}
\|f \|_{H^s_{\eps, \delta} } =   \eps^{\mez} \sum_{k \le s} \eps^{-k}  \|e^{\delta \eps <x>} \D_x^k f \|_{L^2}.
\end{equation}
 Note that for $\delta \le 1$, this norm is equivalent, with constants independent of $\eps$ and $\delta$, 
 to the norm
 $$
\|e^{\delta \eps <x>}  f \|_{H^s_\eps} . 
 $$

\begin{prop}\label{invprop}
Under the assumptions of Theorem \ref{main},  
there are  constants $C$,  $\eps_0 > 0 $ and $\delta_0 > 0$   
and  for all  $\eps \in ]0, \eps_0]$, there is a unit vector $\ell_\eps$ such that 
for $\eps \in ]0, \eps_0]$, $\delta \in [0, \delta_0]$, 
$f \in H^{3}_{\eps, \delta} $, $g \in H^{2}_{\eps, \delta} $ 
the operator equations \eqref{neweq} \eqref{phasecond} has a unique 
solution 
$ U \in H^{2}_{\eps, \delta} $, denoted by $ U = \csLd F$, which satisfies     
  \begin{equation}\label{invbdH2}
\big\|\csLd  F \big\|_{H^2_{\eps, \delta} }\le 
C\eps^{-1}\big( \big\| f \|_{H^{3}_{\eps, \delta} }
+ \big\|g  \big\|_{H^2_{\eps, \delta} }\big).
\end{equation}
 
 Moreover, for  $s \ge 3$, there is a constant $C_s$ such that for  
 $\eps \in ]0, \eps_0]$ 
and 
$f \in H^{s+1}_{\eps, \delta} $, $g \in H^{s}_{\eps, \delta} $ 
the  
solution 
$ U = \csLd F \in H^{s}_{\eps, \delta} $ and 
 \begin{equation}\label{invbdHs}
\big\|\csLd  F \big\|_{H^s_{\eps, \delta} }\le 
C\eps^{-1}\big( \big\| f \|_{H^{s+1}_{\eps, \delta} }
+ \big\|g  \big\|_{H^s_{\eps, \delta} }\big) + C_s  \big\|\csLd  F \big\|_{H^{s-1}_{\eps, \delta} } . 
\end{equation}

\end{prop}

The proof of this proposition  comprises most of the work of the paper.
Once it is established, existence follows by a straightforward
application of the Contraction-Mapping Theorem.
Defining 
\begin{equation}\label{Tdef}
\cT:=\csLd
\begin{pmatrix} 0 \\ - \cR_v + MU + N(U))\end{pmatrix},
\end{equation}
we reduce \eqref{pert} to the fixed-point equation
\begin{equation}\label{fixedeq}
\cT   U:=    U.
\end{equation}


\subsection{Proof of the main theorem}\label{pf}

\begin{proof}[Proof of Theorem \ref{main}]

The profile $\bar u_{NS}$ exists if $\eps$ is small enough. 
The estimates \eqref{NSbds}  imply that 
\begin{equation}
\label{NSbds2}
\| \bar u_{NS} - u_\pm \|_{H^s_{\eps, \delta}}  \le  C_s  \eps 
\end{equation}
with $C_s$ independent of $\eps$ and $\delta$, provided 
that  $\delta \le \theta / 2$. 
Similarly, \eqref{eq:resbds} implies that
\begin{equation}\label{L2resbds}
\| \cR_v\|_{H^s_{\eps, \delta}}\le C_s  \eps^{  3}, 
\end{equation}
and \eqref{defM} implies that 
\begin{equation}\label{Mbds}
\|  M  \|_{H^s_{\eps, \delta}}\le C_s  \eps^{  2}.  
\end{equation}

Moreover, with the choice of norms \eqref{defnorm}, the Sobolev inequality  reads
\begin{equation}
\label{sobemb}
\|  u \|_{L^\infty }  \le  C   \| u \|_{H^1_{\eps}} \le C   \| u \|_{H^1_{\eps, \delta}}
\end{equation}
with $C $ independent of $\eps$. 
Moreover,  for smooth functions $\Phi$, there are nonlinear estimates 
\begin{equation}
\label{nlsest}
\| \Phi (u)   \|_{H^s_{\eps}} \le     C\big(  \|  u \|_{L^\infty } \big)    \ \| u \|_{H^s_{\eps}} . 
\end{equation}
which also extend to weighted spaces,  for $\delta \le 1$: 
\begin{equation}
\label{nlwsest}
\| \Phi (u)   \|_{H^s_{\eps, \delta}} \le     C\big(  \|  u \|_{L^\infty } \big)    \ \| u \|_{H^s_{\eps, \delta}} . 
\end{equation}

In particular, this implies that  for $s \ge 1$, $\delta \le \min\{ 1, \theta/ 2 \}$ and 
$\eps$ small  enough:  
\begin{equation}\label{Mbds2}
\begin{aligned}
\|  M  U   \|_{H^s_{\eps, \delta}}& \le C  \big(  \|  M  \|_{H^1_{\eps, \delta}}     \| U \|_{H^s_{\eps, \delta}}
+    \|  M  \|_{H^s_{\eps, \delta}}     \| U \|_{H^1_{\eps, \delta}}    \big) 
\\
& \le     \eps^{ 2 }   \big(  C     \| U \|_{H^s_{\eps,\delta}}   +  C_s     \| U \|_{H^1_{\eps,\delta}} \big)   
\end{aligned}
\end{equation}
where the first constant $C$ is independent of $s$. 
Similarly, 
 \begin{equation}\label{Nbds2}
\|  N (  U)    \|_{H^s_{\eps, \delta}}\le    C\big(  \|  U \|_{L^\infty } \big)      
 \| U \|_{H^1_{\eps, \delta}}  
 \| U \|_{H^s_{\eps, \delta}}  . 
\end{equation}

Combining these estimates,  we find that
\begin{equation*}
\begin{aligned}
\|\cT U\|_{H^s_{\eps, \delta} } \le   
 \eps^{-1} \big( C_s   \eps^{3 }  +   C  \eps^{2  }    \|U\|_{H^s_{\eps, \delta} } 
 + C_s  \eps^{2  }    \|U\|_{H^1_{\eps, \delta} }  
+  C  \|U\|_{H^1_{\eps, \delta} } \|U\|_{H^s_{\eps, \delta} }  \big), 
\end{aligned}
\end{equation*}
that is
\begin{equation}\label{Tbd}
\begin{aligned}
\|\cT U\|_{H^s_{\eps, \delta} } \le   
  C_s   \eps^{2 }  +   C  ( \eps^{  } +  \eps^{-1}\| U \|_{H^1_{\eps, \delta}} )   \|U\|_{H^s_{\eps, \delta} } 
 + C_s  \eps    \|U\|_{H^1_{\eps, \delta} }  
\end{aligned}
\end{equation}
provided that $\eps \le \eps_0$,  $\delta \le \min\{1, \theta / 2\}$ and 
$\| U \|_{L^\infty} \le 1$.

\medbreak

Consider first the case $s = 2$. 
Then,   $\cT$  maps the   ball 
$\cB_{\eps, \delta} = \{  \| U \|_{H^{2}_{\eps, \delta}}\le  \eps^{1+\frac{1}{2}}\} $ 
    to itself,  if  $\eps \le \eps_1 $ where  $\eps_1 > 0$ is small enough. 
  Similarly, 
  \begin{equation}\label{dTbd}
\begin{aligned}
\|\cT U -\cT V\|_{H^2_{\eps, \delta}}   \le 
C\eps^{-1} \big(\eps^{2} +  \|U\|_{H^2_\eps}  +  \|V\|_{H^2_\eps}\big)  \|U-V\|_{H^2_{\eps, \delta}},\\
\end{aligned}
\end{equation}
provided that  $\| U \|_{L^\infty} \le 1$ and $\| V \|_{L^\infty} \le 1$, 
from which we readily find   
that, for $\eps>0$ sufficiently small,
$\cT$ is contractive on   $\cB_{\eps, \delta}$, whence, by the Contraction-Mapping Theorem,
there exists a unique solution $  U^\eps$ of  \eqref{fixedeq} 
in $\cB_{\eps, \delta}$ for $\eps$ sufficiently small.

Moreover, from the contraction property 
$$
\|\bar U^\eps-\cT(0)\|_{H^2_\eps}= 
\|\cT(\bar U^\eps)-\cT(0)\|_{H^2_\eps}\le
c \|\bar U^\eps \|_{H^2_\eps},
$$
with $ c <1$, we obtain as usual that
$\|  U^{\eps, \delta} \|_{H^2_{\eps, \delta} }\le C\|\cT(0)\|_{H^2_{\eps, \delta}}$,
whence 
\begin{equation}
\|    U^\eps\|_{H^2_{\eps, \delta}  }\le C\eps^{2}. 
\end{equation}
by \eqref{Tbd}.
In particular,  $e^{\eps \delta  \la x \ra } U^\eps =  O( \eps^{2})$ in $H^2_{\eps}$ 
and by the Sobolev embedding 
\begin{equation}
\|  e^{\eps \delta  \la x \ra } U^\eps \|_{L^\infty}   =  O( \eps^{2}), 
\quad \|  e^{\eps \delta  \la x \ra } \D_x U^\eps \|_{L^\infty}   =  O( \eps^{3}). 
\end{equation}

\medbreak 

For $s \ge 3$, the estimates \eqref{Tbd} show that for $\eps \le \eps_1$ independent 
of $s$, the iterates $\cT^n (0)$ are bounded in $H^{s}_{\eps, \delta}$,   
and similarly that $\cT^n(0) - \cT (0) = O (\eps^2)$ in $H^{s}_{\eps, \delta}$, 
implying that the limit $U $ belongs to  $H^{s}_{\eps, \delta}$ with norm 
$O(\eps^2)$. 
Together with the Sobolev inequality \eqref{sobemb}, this implies the pointwise estimates  
\eqref{finalbds}.

Finally, the assertion about uniqueness follows by
uniqueness in $\cB_{c\eps, \delta}$ for the choice $\delta=0$
and $c>0$ sufficiently small
(noting by our argument that also $\cB_{c\eps, \delta}$ is mapped to 
itself for $\eps$ sufficienty small, for any $c>0$),
together with the observation that phase condition
\eqref{phasecond} may be achieved for any solution $\bar U=\bar U_{NS}+U$
with 
$$
\|U\|_{L^\infty}\le c \eps^{2}<< \bar U_{NS}'(0)\sim \eps^2
$$
by translation in $x$, yielding
$\bar U_a(x):=\bar U(x+a)= \bar U_{NS}(x)+ U_a(x)$
with 
$$
U_a(x):= \bar U_{NS}(x+ a)-\bar U_{NS}(x)+ U(x+a) 
$$
so that
$ U_a(0) \sim (a+o(1))\bar U_{NS}'(0) $
and
$ \ell _\eps \cdot u_a(0)  \sim \ell_\eps \cdot \bar u_{NS}'(0)$,
which may be set to zero by appropriate choice of $a$, by
the property $\ell_\eps \cdot \bar u_{NS}'(0)\ne 0$
following from our choice of $\ell_\eps$ (see Remark \ref{ellchoice}).
\end{proof}

It remains to prove existence of the linearized solution
operator and the linearized bounds \eqref{invbdHs}, which
tasks will be the work of the rest of the paper.
We concentrate first on estimates, and prove the existence next, using 
a viscosity method.

%

\section{Internal and high frequency estimates}
 \label{energy}

We begin by establishing a priori estimates on solutions
of the equation \eqref{neweq}
This will be done in two stages.
In the first stage, carried out in
this section, we establish energy estimates
showing that ``microscopic'', or ``internal'', variables consisting
of $v $ and derivatives of $(u, v)$ are controlled by  and small 
with respect to the ``macroscopic'', or ``fluid'' variable, $u$.
In the second stage, carried out in Section \ref{linCEestimates}, we
estimate the macroscopic variable $u$ by Chapman--Enskog approximation
combined with finite-dimensional ODE techniques such as have been
used in the study of fluid-dynamical shocks; see,
for example, \cite{MaZ4, MaZ5, Z1, Z2, GMWZ}.

\subsection{The basic $H^1$ estimate} 
  
  We consider the equation
  \begin{equation}
  \label{inteqs6}
  \cL_*^\eps U := \begin{pmatrix} A_{11} u + A_{12} v  
  \\
   A_{21} u' + A_{22} v'  -   dq (\bar u_{NS} , v_* (\bar u_{NS}))  U \end{pmatrix} = 
  \begin{pmatrix} f \\ g \end{pmatrix}  
  \end{equation}
   and its differentiated form: 
\begin{equation}\label{apriorieq}
AU'- dQ(\bar u_{NS}, v_*(\bar u_{NS}))U=
\begin{pmatrix} f'\\g \end{pmatrix}.
\end{equation}
The internal variables are $U' = (u', v')$ and $\tilde v$ where 
\begin{equation}
\label{tildev}
\tilde v:= 
v   +    p   u  , \qquad 
p =  \partial_v q^{-1}\partial_uq (\bar u_{NS}, v_*(\bar u_{NS})) =  
-  dv_* (\bar u_{NS})
\end{equation} 
is the linearized version of $\bar v-v_*(\bar u)$.

\begin{prop}\label{energypropL2}
Under the assumptions of Theorem \ref{main}, for  there  are  constants 
$C$, $\eps_0 > 0$ and $\delta_0 > 0$ such that for  $0 < \eps \le \eps_0$ and 
$0 \le \delta \le \delta_0$, 
$f \in H^{2}_{\eps, \delta} $, $g \in H^{1}_{\eps, \delta} $ 
and      $U= (u,v)\in H^1_{\eps, \delta}$ of \eqref{inteqs6} satisfies
  \begin{equation}\label{invbd}
\big\| U'   \big\|_{L^2_{\eps, \delta} }   + \big\| \tilde v   \big\|_{L^2_{\eps, \delta} }  \le 
C    \big( \big\| (f, f', f'', g, g') \|_{L^2_{\eps, \delta} }
 + \eps  \big\| u  \big\|_{L^2_{\eps, \delta} }  \big).
\end{equation}

\end{prop}

Making the change of variables $v  \mapsto \tilde v$,  or $U \mapsto  \widetilde U = (u, \tilde v)$ and denoting 
$U=P(\bar u_{NS})\widetilde U$, we obtain an ODE
\begin{equation}\label{symmode}
\widetilde  A \widetilde U'-\widetilde{ Q}\, \widetilde U= \widetilde F + \widetilde  C \widetilde U,
\end{equation}
where 
\begin{equation}
\label{hateq}
\widetilde A=P^{-1}AP,\quad  \widetilde { Q}=P^{-1}dQP=
\begin{pmatrix}
0 & 0 \\
0 &\widetilde{Q}_{22} \end{pmatrix},
\end{equation}
with $ \widetilde{Q}_{22}=  \partial_v q(\bar u_{NS}, v_*(\bar u_{NS})$, 
\begin{equation}
\label{hatF}
\widetilde F=
\begin{pmatrix}
f' \\
g+  \widetilde R_{21} f' 
\end{pmatrix},
\end{equation}
with $ \widetilde R_{21} =  \partial_v q^{-1}\partial_uq (\bar u_{NS}, v_*(\bar u_{NS})  $, 
and
\begin{equation}\label{Cest}
\widetilde  C=-P^{-1}AP'= 
\begin{pmatrix} \widetilde  C_{11} & 0 \\ \widetilde  C_{21} & 0 \end{pmatrix} = O(\bar u_{NS}')  =  \eps^2 \widehat C . 
\end{equation}
\medbreak
The equation \eqref{symmode} reads 
\begin{equation}
\label{simplfeqtilde}
\widetilde A \widetilde U'   -    \widetilde Q \widetilde  U = \widehat  F = 
\begin{pmatrix}
f'    +   \eps    h   \\
\tilde  g  
\end{pmatrix}
\end{equation}
with 
\begin{equation}
\label{def610}
h = \eps \widehat C_{11} u , \qquad \tilde g =  g +  \widetilde R_{21} f' + \eps^2 \widehat C_{21} u. 
\end{equation}

\medbreak

We first prove the estimate \eqref{invbd} for $\delta = 0$. 
Dropping hats and tildes, the  ODE reads
\begin{equation}
\label{simplfeq}
AU' -   Q  U = F, 
\quad
Q =
\begin{pmatrix}
0 & 0 \\
0 &   Q_{22}
\end{pmatrix},
\quad
F =
\begin{pmatrix}
f'    +   \eps    h   \\
g 
\end{pmatrix}. 
\end{equation}
The matrix $A = A(x)$ has end points values $A_\pm$ at $\pm \infty$ and satisfies 
estimates 
\begin{equation}
\label{estcoeff}
| \D_x^k (A - A_\pm)| \le C_k \eps^k 
\end{equation}
with $C_k$ independent of $\eps$. There are similar estimates for $Q_{22}$. 
Moreover, $A$ and $ {Q}$ are simultaneously symmetrizable
by some $  S  =  \widetilde  S(\bar u_{NS})$,
since this property is unaffected by coordinate changes.
$S$ is necessarily block-diagonal,  $SA$ and 
$$
SQ=
\begin{pmatrix}
0 & 0 \\
0 & q
\end{pmatrix}
$$
are symmetric with $q$ negative  definite. 
Likewise, the genuine coupling condition still holds,
which, by the results of \cite{K}, is equivalent to the 
{\it Kawashima condition}, and 
there is a smooth $  K = \widetilde K(\bar u_{NS})= - \widetilde K^* $ such that 
$
\Re  (  K   A  -    S  {Q}) 
$ is definite positive. Therefore, there is  $c > 0$ such that for all 
$\eps \le \eps_0$ and $x \in \RR$: 
\begin{equation}
\label{infbds}
\tilde  q \le   - c \Id, \qquad 
\Re (  K   A -   S  {Q}) \ge c\Id.
\end{equation}

\begin{lem}
\label{lem62}
There is a constant $C$ such that for $\eps$ sufficiently small, 
$f \in H^2$, $g \in H^1$, $h\in H^1$ and  $U\in H^1$  satisfying \eqref{simplfeq}, 
one has 
\begin{equation}\label{sharph1eq}
\| U'  \|_{L^2} + \| v \|_{L^2}  \le C \big(     
\|  f  \|_{H^2}   +   \| h \|_{H^1} + \| g  \|_{H^1}   
+ \eps \| u \|_{L^2} \big). 
\end{equation}
\end{lem}

\begin{proof}
 
 Introduce the symmetrizer
\begin{equation} 
\label{def615}
\cS = \D_x^2 \circ  S   + \D_x \circ  K   -  \lambda  S .  
\end{equation}
One has 
$$
\begin{aligned}
& \Re \D_x^2  \circ    S  \circ (A\D_x  -      Q)  =    \mez \D_x \circ (S A)' \circ \D_x  -     \D_x   \circ S  Q \circ    \D_x  -  \Re \D_x  \circ ( SQ)'   
\\
& \Re \D_x \circ K  (A\D_x -  Q)  =  \D_x \circ  \Re KA \circ  \D_x  -   \re \D_x \circ K  Q 
\\
&  \Re  S (   A\D_x -    Q )    = \mez  (S A)'   -   S Q .
\end{aligned}
$$
Thus 
$$
\begin{aligned}
\Re \cS  \circ ( A\D_x -    Q)  =  & \D_x \circ (\Re AK  - S   Q) \circ \D_x   +   \lambda S   Q  
\\ 
& + \mez \D_x  \circ (SA)' \circ \D_x -  \mez  \lambda  (S A)'   -   \Re \D_x \circ  (SQ)' -   \Re \D_x \circ K  Q. 
\end{aligned}
$$
Therefore, for  $U \in H^2 (\RR)$,  \eqref{infbds} implies that 
$$
\begin{aligned}
   \Re ( \cS   F , U)_{L^2}   \ge & \ c \| \D_x U \|^2_{L^2} + 
\lambda c  \|   v \|^2_{L^2} 
\\ &- \mez \| (S A)'  \|_{L^\infty}  \big( \| \D_x U \|^2_{L^2} + \lambda  \|  U \|^2_{L^2}\big) 
\\
& - \| (SQ)' \|_{L^\infty} \|  U \|_{L^2} \| \D_x U \|_{L^2}  - \| K   \|_{L^\infty} \|\D_x  U \|_{L^2} 
\|  q v  \|_{L^2} . 
\end{aligned}
$$
Taking
 $$
\lambda =  \frac{2}{c} \| K  \|^2_{L^\infty} \| q \|_{L^\infty} ,
$$
and using that 
\begin{equation}
\label{commest}
  \| (S A)'  \|_{L^\infty} + \| (SQ)'  \|_{L^\infty} = O( \eps^2) 
\end{equation}
yields
$$
\| U' \|^2_{L^2} + \| v \|_{L^2}^2 \lesssim     \Re ( \cS   F , U)_{L^2}   + 
\eps^2 \big( \| U \|^2_{L^2} +  \| U' \|^2_{L^2} \big).  
$$

In the opposite direction, using the block structure of $S$, 
$$
\begin{aligned}
   \Re ( \cS F , U)_{L^2}   \le      & \| \D_x  U \|_{L^2} \big( 
 \| \D_x (SF) \|_{L^2}  + \| K \|_{L^\infty} \| F \|_{L^2} \big) 
\\
& + \lambda \big(\eps    \| S_{11} \|_{L^\infty}   \|  u \|_{L^2}  
\|  h  \|_{L^2}  +    \|  ( S_{11} u) '  \|_{L^2}   \|  f \|_{L^2}  
\\
& \hskip 4cm +  \| S_{22} \|_{L^\infty}  \|  v \|_{L^2}   \| g  \|_{L^2} \big).  
\end{aligned}
$$
Using again    
that the derivatives of the coefficients are $O(\eps^2)$,  this implies  that 
\begin{equation*} 
\begin{aligned}
 \Re ( \cS F , U)_{L^2}  \lesssim  & 
 \big(   
\|  f  \|_{H^2}   +   \| h \|_{H^1} + \| g \|_{H^1} \big) \| U' \|_{L^2}  
\\ 
 &+    \eps  \| h \|_{L^2}  \| u \|_{L^2}  +  \eps^2 \| f \|_{L^2} \| u \|_{L^2}  + \| g \|_{L^2} \| v \|_{L^2},
\end{aligned}
\end{equation*}
The estimate \eqref{sharph1eq}  follows provided that $\eps$  is small enough. 

This proves the lemma under the additional assumption that $U \in H^2$. 
When $U\in H^1$, the estimates follows using Friedrichs mollifiers. 
\end{proof}

\begin{proof}[Proof of Proposition~\ref{energypropL2}]
Consider the system \eqref{simplfeqtilde}
Because the coefficients are functions of $\bar u_{NS}$ and its derivatives, there holds   
  $$
\begin{aligned}
\| h, h'\|_{L^2}   &  \le C \eps \|  u'\|_{L^2} 
+ \eps \| u\|_{L^2}
\\
\| \tilde g, \tilde g'\|_{L^2}  &  \le  C\big(  \| f', f'', g, g' \|_{L^2}  
+   \eps^2  \| u'\|_{L^2}   \big) 
\end{aligned}
$$
and 
 $$
\|  v '\|_{L^2}        \le   
\big\| \tilde v'   \big\|_{L^2 }   +   C \big(  \big\|   u'   \big\|_{L^2}
+  \eps^2 \| u\|_{L^2} \big). 
$$
Therefore, the bounds \eqref{sharph1eq}  for  $(\widetilde U', \tilde v)$ 
imply that 
   \begin{equation}\label{invbdtilde}
\big\|  U'   \big\|_{L^2 }   + \big\| \tilde v   \big\|_{L^2}  \le 
C    \big( \big\| (f, f', f'', h, h',  \hat g, \hat g') \|_{L^2 }
 + \eps  \big\| u  \big\|_{L^2 }  \big).
\end{equation}
Multiplying by $\eps^\mez$ and using the estimates of $h$ and $g$ above, yields  \eqref{invbd} for $\delta = 0$. 

For $\delta > 0$ small, consider 
$U ^w =  e^{ \eps \delta \la x \ra } U$. Then, $U^w$ satisfies 
 \begin{equation}
  \label{inteqs6w}
  \cL_*^\eps U^w  = 
  \begin{pmatrix} f^w \\ g^w    \end{pmatrix}  ,   
  \end{equation}
  with $f^w =   e^{ \eps \delta \la x \ra }  f $ and $g^w  =   e^{ \eps \delta \la x \ra }  g + \eps \delta   \la x\ra ' (A_{21} u^w + A_{22}  v^w ) $. 
  We note that, 
  $$
  \|  U'  \|_{L^2_{\eps, \delta}}  \le    \|  (U^w)'  \|_{L^2_{\eps}}  +   \eps \|   U^w  \|_{L^2_{\eps}} , 
, \quad 
    \|  \tilde v  \|_{L^2_{\eps, \delta}} \lesssim   \|  \tilde v^w  \|_{L^2_{\eps}} , 
  $$
  $$
    \|   f^w, (f^w)', (f^w)''   \|_{L^2_{\eps}} \lesssim   \|  (f, f', f'')  \|_{L^2_{\eps, \delta}} , 
$$
$$
\begin{aligned}
     \|   g^w, (g^w)'  \|_{L^2_{\eps}}   \lesssim   \|  (g, g') \|_{L^2_{\eps, \delta}}  + 
    \eps  \delta \| (U, U') \|_{L^2_{\eps, \delta}}.  
  \end{aligned}
  $$
   We use the estimate \eqref{invbd} with $\delta = 0$ for 
 $U^w$, and the Proposition follows provided that $\delta$ is small enough.
\end{proof}

\subsection{Higher order estimates} 

\begin{prop}\label{estHs}  There are constants $C$, $\eps_0 > 0$, $\delta_0 > 0$ and for 
all $k \ge 2$, there is  $C_k$, such that 
 $0< \eps \le \eps_0$, $\delta \le \delta_0$,  
 $U \in H^s_{\eps, \delta}$, 
  $f \in H^{s+1}_{\eps, \delta}$ and $g \in H^{s}_{\eps, \delta}$ 
  satisfying \eqref{simplfeq} satisfies:
\begin{equation}\label{sharph2eq}
\begin{aligned}
  \| \D_x^k U' \|_{L^2_{\eps, \delta}} + 
&  \| \partial^k_x \tilde v \|_{L^2_{\eps, \delta}}  
  \le  C  
  \| \partial_x^k (f, f', f'', g, g')  \|_{L^2_{\eps, \delta} }     \\
  &  +    \eps^k  C_k \big( 
 \| U' \|_{H^{k-1}_{\eps, \delta} } + \eps   \| \tilde v \|_{H^{k-1}_{\eps, \delta} }  +
  \eps \| u \|_{L^2_{\eps, \delta}}\big)  
\end{aligned}
\end{equation}

\end{prop}

\begin{proof} 
Differentiating \eqref{inteqs6} $k$ times, yields
\begin{equation}\label{Dapriorieq}
A\partial_xU^{k} - dQ(\bar u_{NS}, v_*(\bar u_{NS}))\partial_x U^{k} =
\begin{pmatrix} \partial^k_x f'    \\ \partial^{k}_xg +   r_k  \end{pmatrix},
\end{equation}
where 
$$
r_k  =  - \sum _{l=0}^{k-1} \partial^{k-l}_x  Q_{22} \,  \D_{x}^{l} \tilde v. 
$$
Here we have used that $d q (\bar u_{NS}, v_* (\bar u_{NS}) U = Q_{22} \tilde v$. 
The   $H^1$ estimate 
  yields
$$
\begin{aligned}
\| \D_x^k U' \|_{L^2_{\eps, \delta}} + 
 \| \partial^k_x  v + p \D_{x}^k u  \|_{L^2_{\eps, \delta}}   \le  
  C  \big( & \| \partial_x^k (f, f', f'', g, g')   \|_{L^2_{\eps, \delta}}     \\
     + \eps  \| \partial_x^k  u \|_{L^2_{\eps, \delta}}  
 & +   \| \partial_x r_k \|_{L^2_{\eps, \delta}} 
+   \| r_k  \|_{L^2_{\eps, \delta}} \big) , 
\end{aligned}
$$
for $0 \le k \le s$, with  $r_0 = 0$ when $k = 0$. 
Since $Q$ is a function of $\bar u_{NS}$,   its   $k- l$-th derivative   is $O (\eps^{k - l+1})$
when $k-l > 0$. 
Therefore: 
$$
 \| \partial_x r_k \|_{L^2_{\eps, \delta}} 
+   \| r_k  \|_{L^2_{\eps, \delta}} \le C_k  \eps^{k}  \big(  \|\tilde  v' \|_{H^{k-1}_{\eps, \delta}} +
\eps   \| \tilde v \|_{L^2_{\eps, \delta} } \big). 
$$
Similarly,  for $k = 1$
$$
 \| \partial_x \tilde v_k \|_{L^2_{\eps, \delta}} 
 \le  \| \partial_x  v +  p \D_{x} u  \|_{L^2_{\eps, \delta}}  +  C 
 \eps^2  \| u  \|_{L^2_{\eps, \delta}}   
$$
and for $k \ge 2$: 
$$
 \| \partial^k_x \tilde v_k \|_{L^2_{\eps, \delta}} 
 \le  \| \partial^k_x  v + p \D_{x}^k u  \|_{L^2_{\eps, \delta}}  +  C_k 
(   \eps^k \| u'   \|_{H^{k-2}_{\eps, \delta}} +  \eps^{k+1}    \|\tilde  u  \|_{L^2 _{\eps, \delta}} \big). 
$$

\end{proof}


\section{Linearized Chapman--Enskog estimate} \label{linCEestimates}

\subsection{The approximate equations}

It remains only to estimate $\|u\|_{L^2_{\eps, \delta}}$      in order to close the estimates
and establish \eqref{invbd}.
To this end, we work with the first equation  in 
\eqref{inteqs6}
and  estimate it by comparison with the Chapman-Enskog 
approximation (see the computations Section~\ref{CEapprox}). 
  
From the second equation
$$
A_{21}u'+A_{22}v' -g=\partial_u q u+ \partial_v q v= \partial_vq \tilde v,
$$
where we use the notations $\tilde v$ of Proposition~
\ref{energypropL2}, 
we find 
\begin{equation}
\label{T2s7}
\tilde v=\partial_v q^{-1}
\Big((A_{21} + A_{22}\partial_v dv_* (\bar u_{NS}))u ' +A_{22} \tilde v  ' -g
\Big). 
\end{equation}
Introducing $\tilde v$ in the first equation, yields 
$$
(A_{11} + A_{12} d v_* (\bar u_{NS} ) ) u  + A_{12} \tilde v =  f,  
$$ 
 thus
$$
(A_{11} + A_{12} dv_* (\bar u_{NS}) ) u' =  f' - A_{12} \tilde v' -
d^2 v_* (\bar u_{NS}) (\bar u'_{NS}, u) . 
$$
 Therefore, \eqref{T2s7} can be modified to 
 \begin{equation}
 \label{T2bs7}
 \tilde v  =  c_* (\bar u_{NS}) u'   +   r 
\end{equation}
 with 
 $$
 \begin{aligned}
 r = d^{-1}_vq (\bar u_{NS}, &v_* (\bar u_{NS})) \Big(  A_{22}(\tilde v)' -g
 \\
& + dv_* (\bar u_{NS}) \big(  f' - A_{12} \tilde v' -
d^2 v_* (\bar u_{NS}) (\bar u'_{NS}, u)\big) \Big) . 
 \end{aligned}
 $$
This implies that $u$ satisfies the linearized profile equation
\begin{equation}\label{intpertu}
\begin{aligned}
\bar b_* u'- \bar {df}_* u  =   A_{12} r  - f 
\end{aligned}
\end{equation}
where $\bar b_*=b_*(\bar u_{NS})$ 
and $\bar {df}_{*} := df_*(\bar u_{NS}) = A_{11} + A_{12} dv_* (\bar u_{NS})$.


\subsection{$L^2$ estimates and proof   of the main estimates}

\begin{prop}\label{uprop}
The operator $\bar b_*\partial_x -\bar{df}_*$ has
a right inverse $(b_*\partial_x -df^*)^{\dagger}$ 
satisfying
\begin{equation}\label{rightinv}
\|(\bar b_*\partial_x -\bar {df}_*)^{\dagger}h\|_{L^2_{\eps, \delta}} \le 
C\eps^{-1}\|h\|_{L^2_{\eps, \delta}},
\end{equation}
uniquely specified by the property that the solution 
$u = (b_*\partial_x -df^*)^{\dagger} h$  satisfies 
\begin{equation}\label{phase}  
\ell_\eps  \cdot u(0) =0. 
\end{equation}
for certain unit vector $\ell_\eps$. 
\end{prop}

Taking this proposition for granted, we finish the proof of the main estimates in Proposition~\ref{invprop}.

 \begin{prop}
 \label{prop72}
There are  constants $C$,  $\eps_0 > 0 $ and  $\delta_0 > 0$   
such that 
for $\eps \in ]0, \eps_0]$, $\delta \in [0, \delta_0]$,       
$f \in H^{3}_{\eps, \delta} $, $g \in H^{2}_{\eps, \delta} $ and 
$U \in H^2_{\eps, \delta}$ satisfying  
  \eqref{neweq} and \eqref{phasecond}  
  \begin{equation}
  \label{invbdH2s7}
\big\| U \big\|_{H^2_{\eps, \delta} }\le 
C\eps^{-1}\big( \big\| f \|_{H^{3}_{\eps, \delta} }
+ \big\|g  \big\|_{H^2_{\eps, \delta} }\big).
\end{equation}
\end{prop}

\begin{proof}
Going back now to \eqref{intpertu}, $u$ satisfies 
$$
\begin{aligned}
\bar b_* u'- \bar {df}_* u &=  O(|\tilde v'|+ |g| + |f'| + \eps^2 | u |  )  -   f,
\end{aligned}
$$
   If in addition  $u$ satisfies the condition \eqref{phase}
then  
\begin{equation}
\label{temp2}
\|u\|_{L^2_{\eps, \delta}}\le C   \eps^{-1} 
( \|\tilde v'\|_{L^2_{\eps, \delta} }
+ \|(f, f',g)\|_{L^2_{\eps, \delta} }     + \eps^2  \| u \|_{L^2_{\eps, \delta}} \big) . 
\end{equation}

By  Proposition~\ref{energypropL2} and Proposition~\ref{estHs} for $k = 1$, we have 
  \begin{equation} 
  \label{est77}
\big\| U'   \big\|_{L^2_{\eps, \delta} }   + \big\| \tilde v   \big\|_{L^2_{\eps, \delta} }  \le 
C    \big( \big\| (f, f', f'', g, g') \|_{L^2_{\eps, \delta} }
 + \eps  \big\| u  \big\|_{L^2_{\eps, \delta} }  \big).
\end{equation}
\begin{equation}
\label{est78}
\begin{aligned}
  \|   U'' \|_{L^2_{\eps, \delta}} + &
   \big\| \tilde v '   \big\|_{L^2_{\eps, \delta} }  \le 
   \\
& C    \big( \big\| (f', f'', f''',  g',   g'') \|_{L^2_{\eps, \delta} }
 + \eps  \big\| U'  \big\|_{L^2_{\eps, \delta} }   + 
 \eps^2   \big\| u  \big\|_{L^2_{\eps, \delta} } \big).
\end{aligned}
 \end{equation}
Combining these estimates,  
this implies 
 \begin{equation*} 
 \begin{aligned}
  \big\| \tilde v '   \big\|_{L^2_{\eps, \delta} }  & \le    
C      \big( \big\| (f', f'', f''',  g',   g'') \|_{L^2_{\eps, \delta} } + \eps \big\| (f, f', f'', g, g') \|_{L^2_{\eps, \delta} }
 + \eps^2  \big\| u  \big\|_{L^2_{\eps, \delta} }  \big)\\
& \le    
C      \big(   \eps \big\| (f, f', f'', g, g') \|_{H^1_{\eps, \delta} }
 + \eps^2  \big\| u  \big\|_{L^2_{\eps, \delta} }  \big).
 \end{aligned}
\end{equation*}
Substituting in \eqref{temp2}, yields 
$$ 
\eps \|u\|_{L^2_{\eps, \delta}} \le C \big(  \|(f, f',g )\|_{L^2_{\eps, \delta}} + 
 \eps \|(f, f',f'',g, g' )\|_{H^1_{\eps, \delta}}    
+ \eps^2 \|  u \|_{L^2_{\eps, \delta}} \big). 
$$
Hence for $\eps $ small,  
\begin{equation}\label{temp3}
\eps \|u\|_{L^2_{\eps, \delta}} \le C \big(  \|(f, f',g )\|_{L^2_{\eps, \delta}} + 
 \eps \|(f, f',f'',g, g' )\|_{H^1_{\eps, \delta}}     \big). 
\end{equation}
 
Plugging this estimate in \eqref{est77} 
 \begin{equation} 
  \label{est711}
\big\| U'   \big\|_{L^2_{\eps, \delta} }   + \big\| \tilde v   \big\|_{L^2_{\eps, \delta} } 
+  \eps  \big\| u  \big\|_{L^2_{\eps, \delta} }  \le 
C     \big\| (f, f', f'', g, g') \|_{H^1_{\eps, \delta} }
 +  \big).
\end{equation}
Hence, with \eqref{est78}, one has 
\begin{equation}
\label{est712}
\begin{aligned}
  \|   U'' \|_{L^2_{\eps, \delta}} + &
   \big\| \tilde v '   \big\|_{L^2_{\eps, \delta} }  \le 
   \\
& C    \big( \big\| (f', f'', f''',  g',   g'') \|_{L^2_{\eps, \delta} }
 + \eps   \big\| (f, f', f'', g, g') \|_{H^1_{\eps, \delta} }\big).
\end{aligned}
 \end{equation}
Therefore, 
 \begin{equation}
\label{est713}
 \big\| U'   \big\|_{H^1_{\eps, \delta} }          
+  \big\| \tilde v   \big\|_{L^2_{\eps, \delta} } 
+  \eps  \big\| u   \big\|_{L^2_{\eps, \delta} }  \le 
  C     \big\| f, f', f'', g, g'  \big\|_{H^1_{\eps, \delta} }  
\end{equation}
The left hand side dominates 
$$
  \big\| U'   \big\|_{H^1_{\eps, \delta} }  + \eps  \big\| U'   \big\|_{L^2_{\eps, \delta} } 
  =  \eps \big\| U'   \big\|_{H^2_{\eps, \delta} } 
  $$
and the right hand side is smaller than or equal to  
$  \big\|  f  \big\|_{H^2_{\eps, \delta} } +  \big\| g   \big\|_{H^1_{\eps, \delta} } $.
The estimate \eqref{invbdH2s7}  follows.   
\end{proof}

Knowing a bound  for $\| u \|_{L^2_{\eps, \delta}}$, Proposition~\ref{estHs} immediately implies

\begin{prop}\label{prop73}
There are  constants $C$,  $\eps_0 > 0 $ and  $\delta_0 > 0$   
and  for  $s \ge 3$  there is a constant $C_s$
such that 
for $\eps \in ]0, \eps_0]$, $\delta \in [0, \delta_0]$,     
 $f \in H^{s+1}_{\eps, \delta} $, $g \in H^{s}_{\eps, \delta} $ and 
$U \in H^s_{\eps, \delta}$ satisfying  
  \eqref{neweq} and \eqref{phasecond}, one has 
 \begin{equation}\label{invbdHs7}
\big\|  U   \big\|_{H^s_{\eps, \delta} }\le 
C\eps^{-1}\big( \big\| f \|_{H^{s+1}_{\eps, \delta} }
+ \big\|g  \big\|_{H^s_{\eps, \delta} }\big) + C_s  \big\|  U  \big\|_{H^{s-1}_{\eps, \delta} } . 
\end{equation}

\end{prop}


\subsection{Proof of Proposition~\ref{uprop}} 
 
By Assumption \ref{goodred}(i), we may assume that there are  linear coordinates  
$u = (u_1, u_2) \in \RR^{n_1} \times \RR^{n_2}$ 
and $h =( h_1, h_2)  \in \RR^{n_1} \times \RR^{n_2}$, with $n_2 = \mathrm{rank} \ b_* (\bar u) $  
 such that   
\begin{equation}
\label{blokbs}
   b_* (\bar u)  = \begin{pmatrix}0 & 0 \\ b_{21} (\bar u) & b_{22} (\bar u) \end{pmatrix}
 \end{equation}
 and $b_{22}(\bar u)$ is uniformly invertible on   $ \cU_*$. 
 Introducing the new variable
 \begin{equation}
 \label{goodu}
 \tilde u_2 = u_2 +   \bar V  u_1, 
 \quad \bar V =  ( b^{22}) ^{-1}  b_{21}   (\bar u_{NS}),
 \end{equation}
 the equation 
$\bar b_* u'- \bar{df}_*u =h$  has the form: 
 \begin{equation}
 \label{blockprofeq}
 \begin{aligned}
\bar a^{11}  u_1 + \bar a^{12}  \tilde u_2 =  h_1,\\
\bar b^{22}  \tilde u_2'   - \bar a^{21}u_1 - \bar a^ {22} \tilde u_2  =     h_2
\end{aligned}
 \end{equation}
 where 
 $$
 \bar a :=   \bar {df}_*     \begin{pmatrix} \Id & 0 \\ - \bar V   & \Id \end{pmatrix} 
 +   \bar b*  \begin{pmatrix} 0 & 0 \\   \bar V '   & 0 \end{pmatrix}. 
 $$
 
 Assumption \ref{goodred}(ii) implies that the left upper corner block 
 $\bar a^{11}$  is uniformly invertible. Solving the first equation
for $u_1$,   we obtain
the reduced nondegenerate ordinary differential equation 
$$
\bar b_*^{22} \tilde u_2'  +   \bar a^ {21}(\bar a^{11})^{-1}
  \bar a^ {12}  \tilde u_2 
- \bar a^{22} \tilde u_2=  h_2  +  \bar a^ {21}(\bar a^{11})^{-1}
h_1
$$
or
\begin{equation}\label{princ}
\begin{aligned}
\check b u_2'-\check a u_2
 = \check h =O(|h_1|+|h_2|). 
\end{aligned}
\end{equation}
 
Note that $\det \bar df_* =\det \bar a^{11} \det \check a$ by
standard block determinant identities, 
so that $\det \check a  \sim \det \bar df_*$ by Assumption \ref{goodred}(ii).
Moreover, as established in \cite{MaZ4},  by  Assumption \ref{profass} and the construction of the profile
$\bar u_{NS}$  
we find that $ m := (\check b)^{-1}\check a$ has the following properties:

 \quad  i)  with $m_\pm$ denoting the end  points values of $m$, there is $\theta > 0$ such that 
  for all $k$ : 
\begin{equation}
\label{est718}
| \D_x^k ( m(x) - m_\pm)  | \lesssim  \eps^{k+1} e^{ - \eps  \theta | x |};       
\end{equation}
   
\quad    ii)   $m (x)$  has a single simple eigenvalue   of order $\eps$, 
dented by $ \eps \mu (x) $,  and  there is $c > 0$ such that 
   for all $x$ and $\eps $ 
the other eigenvalues $\lambda$ satisfy  $| \Re \lambda | \ge c$; 

\quad iii)  the end point values  $\mu_\pm$ of $\mu$  satisfy
\begin{equation}
\label{est719} 
\mu_-  \ge   \alpha    \qquad \mu_+ \le - \alpha  
\end{equation}
for some $\alpha  > 0$ independent of $\eps$.

\smallbreak 
In the strictly parabolic case $\det b_*\ne 0$, this follows
by a lemma of Majda and Pego \cite{MP}.  

 At this point, we have reduced to the case 
\begin{equation}\label{semifinalform}
u_2'- m (x)   u_2 =  O(|h_1|+|h_2|),
 \end{equation}
with $m$ having the properties listed above. 
The important feature is that $m' = O (\eps^2)  << \eps $,  the spectral gap between stable, unstable, and $\eps$-order subspaces
of $m$.  The conditions above imply that there is a matrix 
$ \omega $ such that 
$$
p := \omega ^{ -1} m \omega   =    \blockdiag\{p^+,  \eps  \mu , p-\},
$$
 where the spectrum of $p_\pm$ lies in $ \pm \Re \lambda \ge c $. Moreover, $\omega $ 
 and $p$ satisfies estimates 
 similar to \eqref{est718}. The change of variables 
 $u_2 =  \omega  z$  reduces \eqref{semifinalform}
 to 
 \begin{equation}
 \label{finalform}
 z'  -   p z   =      \omega^{-1} \omega' z    +     O(|h_1|+|h_2|)  . 
 \end{equation}

The equations  $(z^+)'- p^+ z^+ = h^+ $ and
$(z^-)'- p^-z^-= h^- $ either by standard linear theory 
\cite{He} or by symmetrizer estimates as in \cite{GMWZ},  admit unique
  solutions in weighted $L^2$ spaces,  satisfying 
$$
\| e^{  \delta | x |}  z^ \pm\|_{L^2}\le C \| e^{   \delta | x |}  h^\pm \|_{L^2},  
$$
 provided that $\delta$ remains small, typically  $\delta < | \Re p^\pm |$. 

The equation $z_0' -  \eps \mu  z_0  =  h_0 $ may be converted by
the change of coordinates $x\to \tilde x:= \eps x$ to 
\begin{equation}\label{finalz0}
\D_{\tilde x}  \tilde z_0-\tilde  \mu (\tilde x) z_0=   \tilde h_0 (\tilde x)  = \eps^{-1} h_0 (\tilde x/ \eps) ,
\end{equation}
where $\tilde z_0 (\tilde x) = z_0(\tilde x / \eps) $ and $\tilde \mu (\tilde x):=
\mu (\tilde x/\eps)$. By  \eqref{est718} 
$$
| \tilde \mu (\tilde x)- \mu_\pm|\le Ce^{-\theta |\tilde x|}
$$
with $\mu_\pm $ satisfying  \eqref{est719}. 
This equation is underdetermined with index one, reflecting the 
translation-invariance of the underlying equations.
However, the operator $\partial_{\tilde x}-\tilde  \mu$ has a bounded
$L^2$ right inverse $(\partial_{\tilde x}-\tilde  \mu )^{-1}$, as
may be seen by adjoining an additional artificial constraint
\begin{equation}
  \tilde z_0(0) = 0
\end{equation}
 fixing the phase.  This can be seen by solving explicitly the 
 equation  or  applying the gap lemma
of \cite{MeZ3} to reduce the problem to two constant-coefficient
equations on $\tilde x\gtrless 0$, with boundary conditions 
at $z = 0$. 
 We obtain as a result that  
$$
\|e^{   \delta | \tilde x |} \tilde z_0\|_{L^2 }\le C  \| e^{   \delta | \tilde x |} \tilde h_0\|_{L^2}
$$
if $\delta < \min \{\alpha, \theta\}$, which
yields  by  rescaling  
the   estimate
$$
\|e^{ \eps  \delta | x |} z_0\|_{L^2 }\le C\eps^{-1} \| e^{ \eps  \delta | x |} h_0\|_{L^2}
$$

Together with the (better) previous estimates, this gives
existence and uniqueness for the equation
$$
z'  -p z = h , \qquad z_0(0)= 0
$$ with the estimate 
$\| e^{ \eps  \delta | x |}z\|_{L^2}\le C\eps^{-1} \| e^{ \eps  \delta | x |}h\|_{L^2}$.  Because $\omega^{-1} \omega' = O(\eps^2)$, 
this implies that for $\eps$ small enough, the equation \eqref{finalform} with  $z_0(0) = 0$ 
has a unique solution.   Tracing back to the original variables $u$, the condition
$z_0 (0) = 0$ translates into a condition of the form 
$\ell_\eps \cdot u(0) = 0$. Therefore,   
the equation 
$ \bar b_* u' -  \bar d f_* u = h $ has a unique solution such $u$ that  
$ 
\ell_\eps \cdot u(0) = 0
$, which satisfies 
$$
\|e^{ \eps  \delta | x |}u\|_{L^2}\le C\eps^{-1} \|e^{ \eps  \delta | x |} h\|_{L^2}
$$ 
for $\delta$ and $\eps $ small enough, finishing the proof of Proposition~\ref{uprop}.

\begin{rem}\label{goodman}
\textup{
The estimate of Proposition \ref{uprop} may be recognized
as somewhat similar to the estimates of
Goodman \cite{Go} in the time-evolutionary case.
More precisely, the argument is a simplified version of the 
one used by Plaza and Zumbrun \cite{PZ} to show time-evolutionary
stability of general small-amplitude waves.
}
\end{rem}

\begin{rem}\label{trans}
\textup{
The argument of Proposition \ref{uprop} indicates
that the estimate may be improved by factor $\eps$
in transverse modes $z_\pm$.
However, we see no way to use this to improve the overall
estimates on our iteration scheme.
}
\end{rem}

 %
 %
 %
 
\section{Existence for the linearized problem}

The desired estimates \eqref{invbdH2} and \eqref{invbdHs} are given by 
Propositions~\ref{prop72} and \ref{prop73}. 
It remains to prove existence for the linearized problem
with phase condition $u(0)\cdot r(\eps)=0$.
This we carry out using a vanishing
viscosity argument.

Fixing $\eps$, consider in place of $\csL U=F$ the 
family of modified equations
\begin{equation}\label{modeq}
\csLe U:=\csL U - \eta \begin{pmatrix}u'\\v''\end{pmatrix}= F := 
\begin{pmatrix}f \\g\end{pmatrix},
\quad
\ell_\eps \cdot u(0) =0.
\end{equation}
Differentiating the first equation yields 
\begin{equation}\label{modeqd}
A  U'  - d Q (x) U  -   U'' = 
\begin{pmatrix}f' \\g\end{pmatrix},
\quad
\ell_\eps \cdot u(0) =0.
\end{equation}
where  $d Q (x)$ denotes here the matrix $ dQ (\bar u_{NS}, v_* (\bar u_{NS}))$.

\subsection{Uniform estimates} 
We first prove uniform a-priori estimates. 
We denote by   $ \sS  $ the   Schwartz space and for 
$\delta \ge 0$, by $\sS_{\eps \delta} $ the space of functions $u$ such that 
$e^{ \eps \delta \la x \ra } u \in \sS$, with $\la x \ra = \sqrt{1 + x^2}$ as in \eqref{modx}. 

\begin{prop}
\label{lemunifbds}  
There are  constants  $\eps_0 > 0 $, $\delta_0 > 0$   and $\eta_0> 0$, and for 
all $s \ge 2$ a constant $C_s$, 
such that 
for $\eps \in ]0, \eps_0]$, $\delta \in [0, \delta_0]$, $\eta \in ]0, \eta_0]$, 
and  $U  $ and $F$   in  $ \sS_{\eps\delta}(\RR) $, 
satisfying \eqref{modeq}  
  \begin{equation}\label{unifbds83}
\big\|  U  \big\|_{H^s_{\eps, \delta} }\le 
C_s \eps^{-1}\big( \big\| f \|_{H^{s+1}_{\eps, \delta} }
+ \big\|g  \big\|_{H^s_{\eps, \delta} }\big).
\end{equation}

\end{prop} 

\begin{proof}
 The argument of Proposition~\ref{energypropL2}
  goes through essentially unchanged, with
new $\eta$ terms providing additional favorable higher-derivative
terms sufficient to absorb new higher-derivative errors coming
from the Kawashima part.    More precisely,  
consider again the  change of variables 
$v  \mapsto \tilde v = v + p u$,  $p = \D_v q^{-1} \D_u q (\bar u_{NS}, v_* (\bar u_{NS}))$.  
Denoting $ \widetilde U = (u, \tilde v)$ and  
$U=P(\bar u_{NS})\widetilde U$, \eqref{modeq} is transformed to  
\begin{equation}
\label{transmodeq}
\widetilde A \widetilde U' -   \widetilde Q \widetilde  U - \eta \widetilde U'' =  
\begin{pmatrix}
f'    +   \eps    h   \\
\tilde g  
\end{pmatrix}
\end{equation}
with $\widetilde A$, $\widetilde Q$  as in \eqref{hateq}, $h$ given by 
\eqref{def610} and $\tilde g$ now defined by 
\begin{equation*}
 \tilde g =  g +  \widetilde R_{21} f' + \eps^2 \widehat C_{21} u +   \eta (2 p' u' + p'' u) . 
\end{equation*}

Thus we are led to equations of the form \eqref{simplfeq} with the additional 
term $- \eta U''$ in the left hand side.  Using the symmetrizer $\cS $ \eqref{def615}, 
one gains 
$  \eta   \| U'' \|^2_{L^2}  +\lambda \| U' \|^2_{L^2} $ in the minorization 
of $\Re (\cS F, U ) $  and loses commutator terms which are dominated by 
$$
\eta \| S'' \|_{L^\infty} ( \| U'\|^2_{L^2}  + \| U \|_{L^2} \| U' \|_{L^2} ) 
+ \eta \| K \|_{L^\infty}  (  \| U' \|_{L^2} + \| U \|_{L^2}) \| U'' \|_{L^2} ,  
$$
which can be  absorbed by the left hand side  yielding uniform estimates  
\begin{equation}\label{sharph1eq8}
\sqrt \eta \| \widetilde U'' \|_{L^2}   +   \| \widetilde U'  \|_{L^2} + \| \tilde v \|_{L^2}  \le C \big(     
\|  f  \|_{H^2}   +   \| h \|_{H^1} + \| \tilde g  \|_{H^1}   
+ \eps \| u \|_{L^2} \big). 
\end{equation}
Going back to \eqref{modeqd}, this implies uniform estimates of the form 
  \begin{equation}\label{invbd6}
\sqrt \eta \| U''|_{L^2_{\eps, \delta}}  +   \big\| U'   \big\|_{L^2_{\eps, \delta} }   + \big\| \tilde v   \big\|_{L^2_{\eps, \delta} }  \le 
C    \big( \big\| (f, f', f'', g, g') \|_{L^2_{\eps, \delta} }
 + \eps  \big\| u  \big\|_{L^2_{\eps, \delta} }  \big).
\end{equation}
for $\delta = 0$, and next for $\delta \in [0, \delta_0]$ with 
$\delta_0 > 0$ small, as in the proof of Proposition~\ref{energypropL2}.  

\medbreak

When commuting derivatives to the equation, the additional term $ \eta \D_x^2$ brings no new 
term and the proof of Proposition~\ref{estHs} can be repeated without changes, yielding 
estimates of the form 
\begin{equation}\label{est87}
\begin{aligned}
\sqrt \eta \| D_x^k U'' \|_{L^2_{\eps, \delta} }   +  & \| \D_x^k U' \|_{L^2_{\eps, \delta}}   + 
   \| \partial^k_x \tilde v \|_{L^2_{\eps, \delta}}  
\\  
  \le  C  &
  \| \partial_x^k (f, f', f'', g, g')  \|_{L^2_{\eps, \delta} }     
 \\
 &  +    \eps^k  C_k \big( 
 \| U' \|_{H^{k-1}_{\eps, \delta} } + \eps   \| \tilde v \|_{H^{k-1}_{\eps, \delta} }  +
  \eps \| u \|_{L^2_{\eps, \delta}}\big)  . 
\end{aligned}
\end{equation}

\medbreak

Next, applying the Chapman--Enskog argument of
Section \ref{linCEestimates} to the viscous system, we obtain 
in place of \eqref{intpertu} the equation
\begin{equation}\label{viscintpertu}
\begin{aligned}
\bar b_* u'- \bar {df}_* u &=   f +   O(|\tilde v'|+ |g| + |f'| ) 
+\eps^2 O( | u | )    +    \eta O(|u'|+|U''|),
\end{aligned}
\end{equation}
where the final $\eta$ term coming from artificial viscosity 
is treated as a source. One applies  Proposition~\ref{uprop} 
to estimate $ \eps \|u\|_{L^2_{\eps, \delta}}$  by the 
$ L^2_{\eps, \delta}$-norm of the right hand side, 
and continuing as in the proof of Proposition~\ref{prop72}, 
the estimate \eqref{est713} is now replaced by 
 \begin{equation}
\label{est89}
\begin{aligned}
\sqrt \eta    \| U'''\|_{L^2_{\eps, \delta}} &+  \big\| U'   \big\|_{H^1_{\eps, \delta} }          
+  \big\| \tilde v   \big\|_{L^2_{\eps, \delta} } 
+  \eps  \big\| u   \big\|_{L^2_{\eps, \delta} } 
\\
& \le 
  C    \big(   \big\| f, f', f'', g, g'  \big\|_{H^1_{\eps, \delta} }  
  + \eta  ( \| U'\|_{L^2_{\eps, \delta}}  + \| U''\|_{L^2_{\eps, \delta}} ) \big) . 
  \end{aligned}
\end{equation}
  Therefore, for $\eta$ small, the new  $O(\eta)$  terms can be absorbed, and 
  \eqref{unifbds83} for $s = 2$ follows as before. The higher order estimates 
 follow from \eqref{est87}. 
\end{proof} 


\subsection{Existence} 

We now prove existence and uniqueness for \eqref{modeq}. 
First, recast the the problem as a first-order system
\begin{equation}
\label{o1red}
\cU'  - \mA \cU = \cF 
\end{equation}
with 
$$
\cU =  \begin{pmatrix}
u\\v\\v'
\end{pmatrix}'
 , \qquad 
\cF = \begin{pmatrix} f\\0\\g \end{pmatrix},
$$
and 
\begin{equation}\label{AA}
\mA:=
\eta^{-1}
\begin{pmatrix}
 A_{11}& A_{12} & 0\\
0 & 0 & \eta I\\
\eta^{-1} A_{21}A_{11}-   Q_{21} & \eta^{-1} A_{21} A_{12}  -  Q_{22}  & A_{22}\\
\end{pmatrix}.
\end{equation}
Next, consider this as a transmission problem or a doubled boundary value problem on $x\gtrless 0$,
with boundary condtitions given by the $n+2r$ matching conditions $\cU(0^-)=\cU(0^+)$ at $x=0$  together with   the phase condition $\ell_\eps  \cdot u(0)  =0$,  
that is  $n+2r+1$ conditions in all: 
\begin{equation}
\label{transmcond}
\cU(0^-)=\cU(0^+), \qquad  \ell_\eps  \cdot u(0)  =0. 
\end{equation}

Note  that the coefficient matrix $\mA$ converges exponentially to
its endstates at $\pm\infty$. 

\begin{lem}
\label{lem82}  There is  $\theta_1 > 0$ such that for $\eps$ small enough , the matrices 
$\mA_\pm$ have no eigenvalue in the strip 
$| \Re z  |  \le \eps \delta_0$. 
\end{lem}

\begin{proof}
The proof is parallel to the proof of the estimates. Dropping the $\pm$, 
 suppose  that  $i\tau$ is an eigenvalue of $\mA$, or equivalently that 
 there is a constant vector $   U\ne 0$ such that 
  $ e^{i\tau x}  U$ is a solution of 
of equations  \eqref{modeq}
Thus 
\begin{equation}
\label{eq814}
\begin{aligned}
& A_{11}   u + A_{12}   v = i \tau \eta   u,
\\
&(i \tau   A - Q  + \tau^2 \eta )   U = 0 . 
\end{aligned}
\end{equation}

Introduce once again the variable 
$\tilde v = v + Q_{22}^{-1}  Q_{21 } u $, so that the equation is transformed to 
\begin{equation}
\label{eq815}
\begin{aligned}
& A^*_{11}  u + A_{12} \tilde v = i \tau \eta   u,
\\
&(i \tau \tilde A - \tilde Q^\pm + \tau^2 \eta )  U = 0 . 
\end{aligned}
\end{equation}
where $\widetilde A$ and $\widetilde Q$ now denote the end point values of the matrices 
defined at \eqref{hateq}.  Denoting by $\widetilde S$ and $\widetilde K$ the end point values 
of the symmetrizer and Kawashima's multipliers associated  to  $\widetilde A$ and 
$\widetilde Q$,  
consider the multiplier 
$$
\Sigma = | \tau|^2 S  - i \overline \tau K  - \lambda S . 
$$
Multiplying the second equation in \eqref{eq815} by  $\Sigma$ and taking the real part of the 
scalar product with $U$ yields
$$
\begin{aligned}
  | \tau |^2   \Re (\widetilde  K  \widetilde  A & - \widetilde  S  \widetilde Q U, U)  
+ \lambda (\widetilde S \widetilde  Q U, U) + \eta | \tau |^4 (\widetilde  SU, U) 
\\
&\le  C \big( | \im \tau | (| \tau |^2 + \lambda ) \big) | U |^2  +  C  
| \tau |  | \widetilde  Q U| | U | 
\\
& \qquad \qquad  + \eta ( | \tau |^2 |\im \tau|^2  + | \tau|^3 + \lambda | \tau |^2) \big)  | U |^2  . 
\end{aligned}
$$ 
Therefore, choosing appropriately $\lambda$, for $\eta$ and 
$|\im \tau |$ sufficiently small, one has 
\begin{equation}
\label{eq816}
(\eta | \tau |^4 +  | \tau |^2)  | U |^2 + | \tilde v |^2 \le C | \im \tau | | u |^2  
\end{equation}
In particular, $|\tau| $ must be small if $\im \tau $ is small. 

From the equation 
$i \tau \widetilde A_{21} u + \widetilde A_{22} \tilde v - Q_{22} \tilde v + \eta \tau^2 v  = 0 $
one deduces that 
$$
\tilde v -   i \tau (\widetilde Q_{22})^{-1} \widetilde A_{21} u  
=  O (| \tau |  + \eta | \tau |^2 ) | \tilde v | . 
$$
Substituting in the first equation of \eqref{eq815}, we obtain the 
Chapman-Enskog approximation 
\begin{equation*}
( A_{11}^*  - i \tau \bar b_* ) u =   O( \eta | \tau |  + 
  | \tau| + \eta | \tau|^2 )  | \im \tau|^\mez )  ) | u |
\end{equation*}
 where $\bar b_*$ denotes the end point value of the function  \eqref{bstar}.  
 Therefore,    
\begin{equation}
\label{eq817}
| ( \bar b_*)^{-1}  A^*_{11} u  - i \tau u  |  \le  C  |\im \tau|^\mez  |  \tau |  | u | 
\end{equation}
 with arbitrarily small $c >0$. 
 We know from Assumption~\ref{profass} that for 
 $\eps$ small, $( \bar b_*)^{-1}  A^*_{11}$ has a unique small eigenvalue, of order 
 $O(\eps)$, real. Let us denote it by $\eps \mu$. Then we know that 
 $| \mu |$ is bounded from below, see \eqref{est719}. Then 
 \eqref{eq817} implies that there is a constant $C$ such that  for $| \im \tau |$  small enough, 
 and thus $| \tau| $ small,   
 $| i \tau - \eps \mu | \le   C  |\im \tau|^\mez  |  \tau |  $. Therefore,  
 $ l \im \tau + \eps \mu |  \le \mez \eps | \mu |$ if $\eps$ is small enough. 
  
 Summing up, we have proved that if 
 $\eps$ is small enough, $\mA$ has at most one eigenvalue $z $   in the strip 
 $| \re z  \le   \eps 2 | \mu| $,  such that 
 $ |z - \eps  \mu |  \le \mez \eps | \mu|$. 
  This implies the lemma.
 \end{proof}

\begin{rem}
\textup{The same reasoning can be applied to prove that $\mA$ actually has a simple eigenvalue such that  $ |z - \eps  \mu |  \le \mez \eps | \mu|$. }
\end{rem}

\begin{prop}\label{viscexist}
There are  constants  $\eps_0 > 0 $, $\delta_0 > 0$   and $\eta_0> 0$ 
such that 
for $\eps \in ]0, \eps_0]$, $\delta \in [0, \delta_0]$, $\eta \in ]0, \eta_0]$, 
and $F$   in  $ \sS_{\eps\delta}(\RR) $, \eqref{modeq}
admits a unique solution $ U\in \sS_{\eps \delta} (\RR)$. 
\end{prop}

\begin{proof}
 Noting that the coefficient matrix $\mA$ converges exponentially to
 $\mA_\pm$ at   $\pm\infty$,  we may apply the conjugation lemma
of \cite{MeZ1} to convert the equation \eqref{o1red}  by an asymptotically trivial 
change of coordinates $\cU=T(x)Z$ to a constant-coefficient problems
\begin{equation}\label{ccprob}
Z_-' - \mA_- Z_-=F_-,
\quad
Z_+'-\mA_+ Z_+=F_+,
\end{equation}
on $\{ \pm x \ge 0 \}$, 
with $n+2r+1$ modified boundary conditions determined by the value of the
transformation $T$ at $x=0$, where $\mA_\pm:= \mA(\pm \infty)$, and
$Z_\pm(x):= Z(  x)$ for $ \pm x>0$.

By standard boundary-value theory (see, e.g., \cite{He}), to prove existence 
and uniqueness in the Schwartz space for the   problem \eqref{o1red} on 
$\{ x < 0\}$ and $\{ x > 0 \}$ with transmission conditions \eqref{transmcond},  it is sufficient
to show that 

\quad (i) the limiting coefficient matrices $\mA_\pm$ are hyperbolic,
i.e., have no pure imaginary eigenvalues, 

\quad (ii) the number of boundary conditions is equal to the number
of stable (i.e., negative real part)
eigenvalues of $\mA_+$ plus the number of unstable eigenvalues (i.e., positive real part) of
$\mA_-$, and

\quad  (iii) there exists no nontrivial solution of the
homogeneous equation $f=0$, $g=0$.

Moreover, since the eigenvalues of $\mA_\pm$  are located in  $\{ | \re z | \ge \theta_1 \eps$, 
the conjugated form \eqref{ccprob} of the equation show that if the source term $f$ has an exponential decay  $e^{ - \eps \delta \la x \ra }$ at infinity, then the bounded solution also 
has the same exponential decay, provided that 
$\delta < \theta_1$ . Therefore, the three conditions above are also sufficient to prove existence 
and uniqueness in $\sS_{\eps \delta} $ if $\eps$ and $\delta $ are small. 

\medbreak

Note that  (i) is a consequence of Lemma~\ref{lem82}, while (iii) follows from the 
estimate \eqref{unifbds83}. 
 To verify (ii), it is enough to establish the formulae
\begin{equation}\label{dims}
\begin{aligned}
\dim \cS(\mA_\pm)&= r+ \dim \cS(A_{11}^{*\pm}),\\
\dim \cU(\mA_\pm)&= r+ \dim \cU(A_{11}^{*\pm}),
\end{aligned}
\end{equation}
where $A_{11}^{* \pm}  = df_* (u_\pm) = A_{11} + A_{12} dv_*(u_\pm)$ 
and $\cS(M)$ and $\cU(M)$ denote the stable and unstable
subspaces of a matrix $M$. 
We note that $A_{11}^{*\pm} =df_*(u_\pm )$ are invertible, 
with dimensions of the  stable subspace of $A_{11}^{*+}$
and the unstable subspace of $A_{11}^{*-}$ summing to 
$n+1$, by Proposition \ref{NSprofbds}.
Thus, \eqref{dims} implies that 
$$
\dim \cS(\mA_+)+ \dim \cU(\mA_-)= 2r+ \dim \cS(A_{11}^{*+})
+ \dim \cU(A_{11}^{*-})= 2r + n+1
$$
as claimed.

To establish \eqref{dims}, introduce the variable 
$\tilde v = v + Q_{22}^{-1}  Q_{21 } u $, and the variable corresponding 
to $\tilde v'$ scaled by a factor  $\eta^\mez$, that is 
$\tilde w = \eta^\mez  w + \eta^{-\mez}  Q_{22}^{-1}  Q_{21 } ( A_{11} u + A_{12} v) $. 
After this change of variables, the matrix $\mA$ it conjugated   
to $\widetilde \mA$  with 
\begin{equation}
\eta^\mez \widetilde \mA = 
\begin{pmatrix}
 0& 0 & 0\\
0 & 0 &   I\\
0  &   -  Q_{22}  & 0 \\
\end{pmatrix} +  \eta^{- \mez}\begin{pmatrix}
 A^*_{11}& A_{12} & 0\\
0 & 0 & 0 \\
O (\eta^{-\mez})        &   O (\eta^{-\mez})     & A_{22}\\
\end{pmatrix}.
\end{equation}
  From (i),  the matrix $\eta^\mez \widetilde \mA$ has no eigenvelue
  on the imaginary axis, and the number of eigenvalues in 
  $\{ \Re \lambda > 0 \}$ is independent of $\eta$, and thus can be determined
 taking $\eta$ to infinity. 
 The limiting matrix has $r$ eigenvalues in 
 $\{ \Re \lambda > 0 \}$,  $r$ eigenvalues in 
 $\{ \Re \lambda < 0 \}$   and the eigenvalue $0$ with multiplicity $n$, since  
$-Q_{22}$ has its spectrum in   $\{ \Re \lambda > 0 \}$.  
The classical perturbation theory as in \cite{MaZ1} shows that 
for $\eta^{- \mez}  $ small,   $\eta^\mez \widetilde \mA$ has 
$n $ eigenvalues of order $\eta^{- \mez} $, close to the spectrum of 
$A_{11}^* $ with error $O (\eta^{-1})$.  
 Thus, for $\eta > 0$ large,  
 $\eta^\mez \widetilde \mA$ has $r + \dim \cS (A_{11}^*)$ eigenvalue 
 in $\{ \Re \lambda < 0 \}$, proving \eqref{dims}. 
 
  The proof of the Proposition is now complete. 
\end{proof}

\subsection{Proof of Proposition \ref{invprop}}
Let $(\csLe)^\dagger$ denote the inverse operator of 
$\csLe$ defined by \eqref{modeq}, for   sufficiently small $\eta>0$. 
The uniform bound \eqref{unifbds83}, and weak compactness of the
unit ball in $H^2$,  for $F \in \sS$, we obtain existence of a weak solution $U \in H^2$ of
\begin{equation} 
\label{eq820}
\csL U  = F:= \begin{pmatrix}f\\g\end{pmatrix}, \qquad \ell_\eps \cdot u(0 ) = 0,  
\end{equation}
 along some weakly
convergent subsequence.
Proposition~\ref{prop72} implies uniqueness in $H^2$ for this problem,  
therefore the full family converges, giving sense to the definition 
\begin{equation}
\csLd = \lim_{\eta \to 0}  (\csLe)^\dagger
\end{equation}
acting from $\sS$ to $H^2$. 

 For $F \in \cS_{\eps \delta}$, the uniform bounds \eqref{unifbds83} imply that 
 the limit $  \csLd U \in H^s_{\eps, \delta}$ and satisfies same estimate. 
 By density, the operator $\csLd$  extends to 
 $f \in H^{s+1}_{\eps, \delta}$ and $g \in H^{1}_{\eps, \delta}$, with  
 $\csLd F \in H^s_{\eps, \delta}$.

 The sharp  bound \eqref{invbdH2} and \eqref{invbdHs} 
  now follow  immediately from Propositions~\ref{prop72} and 
  \ref{prop73}. The proof of Proposition~\ref{invprop} is now complete.

\begin{rem}
\textup{
We have used freely the finite-dimensionality of
$v$ in our proof of linearized existence.
However, as promised, it plays no role in the final
linearized bounds.
Thus, our result may be used together with discretization
(Galerkin approximation) of $v$ to obtain results also
in the case that $v$ is infinite-dimensional, as we do for
the Boltzmann equations in \cite{MeZ2}.
}
\end{rem}


\section{Application to spectral stability}\label{spectralstab}

\begin{proof}[Proof of Corollary \ref{MaZcor}]
In \cite{MaZ3}, under the same structural conditions assumed here,
it was shown that small-amplitude profiles
of general quasilinear relaxation systems are spectrally stable,
provided that
\begin{equation}\label{sizeder}
   |\bar U'|_{{}_{L^\infty}}\leq C|U_+-U_-|^2,\qquad\qquad
   |\bar U''(x)|\leq C|U_+-U_-|\,|\bar U'(x)|,
\end{equation}
and
\begin{equation}\label{dirder}
   \Big|\frac{\bar U'}{|\bar U'|} +\sgn (\eta) R_0\Big|\leq C\,|U_+-U_-|,
\end{equation}
$$
R_0:=\begin{pmatrix}r(u_0)\\ dv_*(U_0)r(u_0)\end{pmatrix},
$$
where $r(u_0)$ as defined in Theorem \ref{main}
is the eigenvector of $df_*$ at base point $U_0$ in
the principal direction of the shock.
From the bounds of Theorem \ref{main}, we immediately 
verify these conditions, giving the result.
\end{proof}

\end{document}